%% file: paper.tex
\documentclass[a4paper]{article}
\usepackage{geometry}
\usepackage[hidelinks]{hyperref}
\usepackage{graphicx}        %
\usepackage{multicol}        %
\usepackage[bottom]{footmisc}%

\usepackage[font=footnotesize]{caption}
\usepackage{amsthm,amsmath,amssymb}

\newtheorem{proposition}{Proposition}
\newtheorem{remark}[proposition]{Remark}
\newtheorem{theorem}[proposition]{Theorem}
\newtheorem{definition}[proposition]{Definition}

\usepackage{enumerate} 
\usepackage{tikz} 
\definecolor{darkgreen}{rgb}{0,0.6,0}

\usepackage{float}
\usepackage[font=small,labelfont=bf,compatibility=false]{caption}
\usepackage[font=footnotesize]{subcaption}

\usepackage{epstopdf}
\usepackage{graphicx,psfrag,epsf}

\usepackage{paralist}
\floatstyle{ruled}
	\newfloat{algorithm}{h}{loa}
\floatname{algorithm}{Algorithm}
\usepackage{algpseudocode}

\usepackage{tikz}
\usepackage{pgfplots}
\usetikzlibrary{spy}

\newcommand{\deltaSNR}{\mathrm{\Delta SNR}}
\newcommand{\dB}{\mathrm{\,dB}}

\newcommand{\dist}{\mathrm{d}}
\newcommand {\mean}{\mathrm{mean}}
\newcommand {\diam}{\mathrm{diam}}

\newcommand {\pt}{\mathrm{pt}}
\newcommand {\prox}{\mathrm{prox}}
\newcommand {\argmin}{\mathrm{argmin}}
\newcommand {\TV}{\mathrm{TV}}
\newcommand {\TGV}{\mathrm{TGV}}
\newcommand {\STGV}{\mathrm{S}\text{-}\mathrm{TGV}}
\newcommand {\TGVat}{\mathrm{TGV}_\alpha^2}
\newcommand {\STGVat}{\STGV_\alpha^2}
\newcommand{\IC}{\mathrm{IC}}
\newcommand{\ds}{D_{\mathrm{S}}}
\newcommand{\dpt}{D_{\mathrm{pt}}}
\newcommand{\dc}{D_c}
\newcommand{\dcc}{D_{cc}}
\newcommand{\Ac}{\mathcal{A}}
\newcommand{\Nc}{\mathcal{N}}
\newcommand{\Mc}{\mathcal{M}}
\newcommand{\Sc}{\mathcal{S}}
\newcommand{\RR}{\mathbb{R}}
\newcommand{\J}{\mathrm{J}}
\newcommand{\symgrad}{\mathcal{E}}
\newcommand{\sym}{\mathrm{sym}}
\newcommand{\traj}{\mathrm{traj}}

\begin{document}

\title{Non-smooth variational regularization for processing manifold-valued data}
\author{M. Holler and A. Weinmann}

\maketitle

\newcommand{\awei}[2]{{\color{red}AW: #1 }{\color{blue} #2}}
\newcommand{\mh}[2]{{\color{red}MH: #1 }{\color{blue} #2}}

\begin{abstract}
	Many methods for processing scalar and vector valued images, volumes and other data in the context of inverse problems are based on variational formulations. Such formulations require appropriate regularization functionals that model expected properties of the object to reconstruct. Prominent examples of regularization functionals in a vector-space context are the total variation (TV) and the Mumford-Shah functional, as well as higher-order schemes such as total generalized variation models.
	Driven by applications where the signals or data live in nonlinear manifolds, there has been quite some interest in developing analogous methods for nonlinear, manifold-valued data recently. In this chapter, we consider various variational regularization methods for manifold-valued data. In particular, we consider TV minimization as well as higher order models such as total generalized variation (TGV). Also, we discuss (discrete) Mumford-Shah models and related methods for piecewise constant data. We develop discrete energies for denoising and report on algorithmic approaches to minimize them. Further, we also deal with the extension of such methods to incorporate indirect measurement terms, thus addressing the inverse problem setup.
	Finally, we discuss wavelet sparse regularization for manifold-valued data.
\end{abstract}

\section{Introduction}

Any measurement process, either direct or indirect, produces noisy data. While for some setups, the noise can safely be ignored, for many others it severely hinders an interpretation or further processing of the data of interest. In addition, measurements might also be incomplete such that again direct usability of the measured data is limited. 

Variational regularization, i.e., a postprocessing or reconstruction of the quantity of interest via the minimization of an energy functional, often allows to reduce data corruption significantly. The success of such methods heavily relies on suitable regularization functionals and, in particular in the broadly relevant situation that the quantity of interest is sparse in some sense, non-smooth functionals are known to perform very well.
Prominent and well-established examples of non-smooth regularization functional in the context of vector-space data are for instance the total variation functional and higher-order extensions such as total generalized variation, the Mumford-Shah functional and the $\ell^1$- or $\ell^0$-penalization of coefficients w.r.t. some wavelet basis.

When it comes to data in a non-linear space such as a manifold, the situation is different and the development of appropriate analogues of non-smooth regularization functionals in this setting is currently an active topic of research with many challenges still to be overcome. Most of these challenges are related to the nonlinearity of the underlying space, which complicates the transfer of concepts from the  context of vector-space regularizers, such as measure-valued derivatives or basis transforms, but also their numerical realization.

On the other hand, applications where the underlying data naturally lives in a non-linear space are frequent and relevant. A prominent example is diffusion tensor imaging (DTI), which is a technique to quantify non-invasively the diffusional characteristics of a specimen \cite{basser1994mr, johansen2009diffusion}.
Here the underlying data space is the set of positive (definite) matrices, which becomes a Cartan-Hadamard manifold when equipped with the corresponding Fisher-Rao metric.  Another example is interferometric synthetic aperture radar (InSAR) imaging, which is an important airborne imaging modality for geodesy \cite{massonnet1998radar}. Often the InSAR image has the interpretation of a wrapped periodic version of a digital elevation model \cite{rocca1997overview} and the underlying data space is the unit circle ${\mathbb S}^1.$
Further examples are nonlinear color spaces for image processing, as for instance the LCh, HSV and HSL color spaces (where the underlying manifold is the cylinder $\mathbb{R}^2 \times \mathbb S^1$) and chromaticity-based color spaces where the underlying manifold is 
${\mathbb S}^2 \times \mathbb R$, see \cite{chan2001total}.
Also, the rotation group $\mathrm{SO}(3)$ appears as data space in the context of 
aircraft orientations and camera positions \cite{rahman2005multiscale}, protein alignments \cite{green2006bayesian}, and the tracking of 3D rotational data arising in robotics \cite{drummond2002real}.
Data in the euclidean motion group $\mathrm{SE}(3)$ may represent poses \cite{rosman2012group} and sphere-valued data appear as orientation fields of three dimensional images \cite{rezakhaniha2012experimental}. 
Finally, shape-space data \cite{Michor07,berkels2013discrete} constitutes manifold-valued data as well. 

Motivated by such applications, we review existing non-smooth regularization techniques for non-linear geometric data and their numerical realization in this chapter. Following the majority of existing approaches, we will concentrate on discrete signals in a finite difference setting, which is appropriate particularly for image processing tasks due to the mostly Cartesian grid domains of images.
We start with total variation regularization in Section~\ref{sec:TV}, which can be transferred to a rather simple yet effective approach for non-linear data with different possibilities for a numerical realization. With the aim of overcoming well-known drawbacks of TV regularization, in particular so-called staircasing effects, we then move to higher-order functionals in Section~\ref{sec:higherOandTGV}, where the goal is to provide a model for piecewise smooth data with jumps. 
Next, having a similar scope, we discuss different models for Mumford-Shah regularization and their algorithmic realization (using concepts of dynamic programming) in Section~\ref{sec:MumSha}.
Indirect measurements in the context of manifold valued data are then the scope of Section~\ref{sec:InvProb}, where we consider a regularization framework and algorithmic realization that applies to the previously defined approaches. Finally, we deal with 
wavelet sparse regularization of manifold valued data
in Section~\ref{sec:WavSparse} where we consider $\ell^1$ and $\ell^0$ type models and their algorithmic realization.

\section{Total Variation Regularization of Manifold Valued Data}
\label{sec:TV}

For scalar data, total variation regularization was early considered by Rudin, Osher and Fatemi \cite{rudin1992nonlinear} and by Chambolle and Lions \cite{chambolle1997inftv_mh} in the 1990s.
A major advantage of total variation regularization compared to classical Tikhonov regularization is that it preserves sharp edges \cite{strong2003edge,gousseau2001natural} which is the reason for a high popularity of TV regularization in particular in applications with image-related data.
The most direct application of TV regularization is denoising, where $\ell^2$ data terms have originally been used in \cite{rudin1992nonlinear} (and are well-suited in case of Gaussian noise) and $\ell^1$ data terms are popular due to robustness against outliers and some favorable analytical properties \cite{alliney1992digital,nikolova2002minimizers,chan2005aspects}. An extension of TV for vector-valued data has early been considered in \cite{sapiro1996anisotropic} and we refer to \cite{Duran2016color_tv} for an overview of different approaches.

This section reviews existing extensions of TV regularization to manifold-valued data. 
In the continuous setting, such an extension has been considered analytically in \cite{GM06,GMS93}, where \cite{GMS93} deals with the $\mathbb S^1$ case and \cite{GM06} deals with the general case using the notion of cartesian currents. There, in particular, the existence of minimizers of certain TV-type energies in the continuous domain setup has been shown.
In a discrete, manifold-valued setting, there is a rather straight forward definition of TV. 
Here, the challenge is more to develop appropriate algorithmic realizations. Indeed, many of the successful numerical algorithms for TV minimization in the vector space setting, such as \cite{chambolle2004algorithm,chambolle2011first,goldstein2009split} and \cite{nikolova2004variational} for $\ell^1$-TV, rely on smoothing or convex duality, where for the latter no comprehensive theory is available in the manifold setting.

\subsection{Models}
\label{sec:TVmodel}
For univariate data of length $N$ in a finite dimensional Riemannian manifold $\Mc$, the (discrete) TV denoising problem with $\ell^q$-type data fidelity term reads as
\begin{equation} \label{eq:abstract_min_problem_tv_manifold_univariate}
\argmin_{x \in \mathcal M^{N}} \Big\{ \frac{1}{q} \sum_{i =1}^{N}  \dist(x_{i},f_{i})^q +  \alpha \sum_{i=1}^{N-1} \dist(x_{i},x_{i+1})
\Big\}.
\end{equation}
Here, $f=(f_{i})_{i=1}^N$ denotes the observed data and $x=(x_{i})_{i=1}^N$ is the argument to optimize for. Further, $q \in [1,\infty)$ is a real number and $\alpha>0$ is a regularization parameter controlling the trade of between data fidelity and the regularity. 
The symbol $\dist(y,z)$ denotes the distance induced by the Riemannian metric on the manifold $\mathcal M.$
We note that in the euclidean case $\mathcal M=\mathbb R^d,$ the above distance to the data $f$ corresponds to the $\ell^q$ norm. For noise types with heavier tails (such as Laplacian noise in the euclidean case,) $q=1$ is a good choice.
We further point out that, in the scalar case $\mathcal M = \mathbb R$, the expression $\sum_{i=1}^{N-1} d(x_{i},x_{i+1})$ defines the total variation of the sequence $x$ interpreted as a finite sum of point measures. 

In the bivariate case, a manifold version of TV denoising for signals in $\Mc^{N \times M}$ is given by
\begin{align} \label{eq:abstract_min_problem__tv_manifold_bivariate}
\argmin_{u \in \mathcal \Mc^{N\times M}} \Big\{
 \tfrac{1}{q} &\sum\nolimits_{i,j}  \dist(x_{i,j},f_{i,j})^q 
\\&+  \alpha \sum\nolimits_{i,j} \left( \dist(x_{i,j},x_{i+1,j})^p + \dist(x_{i,j},x_{i,j+1})^p\right)^{1/p} \Big\} .
\notag
\end{align}
Note that here and on the following, we will frequently omit the index bounds in finite-length signals and sums for the sake of simplicity, and always implicitly set all scalar-valued summands containing out-of-bound indices to $0$.
In \eqref{eq:abstract_min_problem__tv_manifold_bivariate}, the cases $p=1$ and $p=2$ are most relevant, where $p=1$ has computational advantages due to a separable structure and $p=2$ is often used because is corresponds to an isotropic functional in the continuous, vector-space case. We note however that, in the TV case, the effects resulting from anisotropic discretization are not severe. Moreover, they can be almost completely eliminated 
by including further difference directions, such as diagonal differences. For details on including 
further difference directions we refer to Section~\ref{sec:MumSha} 
(discussing reduction of anisotropy effects for the Mumford-Shah case in which case such effects are more relevant.)

Note that if we replace the distance term in the TV component of \eqref{eq:abstract_min_problem_tv_manifold_univariate}
by the squared distance
(or remove the square root for $p=2$ in \eqref{eq:abstract_min_problem__tv_manifold_bivariate}), we end up with a discrete model of classical $H^1$ regularization. Further, we may also replace the distance term in the regularizer by $h \circ \dist$ where $h$ can for instance be the so-called Huber function which is a parabola for small arguments smoothly glued with two linear functions for larger arguments. Using this, we end up with models for Huber regularization, see \cite{weinmann2014total} for details.

\subsection{Algorithmic Realization}
\label{sec:TValgo}
As mentioned in the introduction to this section,
the typical methods used for TV regularization in vector spaces are based on convex duality. The respective concepts are not available in a manifold setting.
However, there are different strategies to solve \eqref{eq:abstract_min_problem_tv_manifold_univariate} and \eqref{eq:abstract_min_problem__tv_manifold_bivariate}, and we briefly review some relevant strategies in the following.

The authors of \cite{SC11,CS13} consider TV regularization for $\mathbb S^1$-valued data and develop a lifting approach, i.e., they lift functions with values in $\mathbb S^1$ to functions with values in the universal covering $\mathbb R$ of $\mathbb S^1$, lifting the involved functionals at the same time such that periodicity of the data is respected. This results in a nonconvex problem for real valued data (which still reflects the original $\mathbb S^1$ situation), which can then algorithmically be approached by using convex optimization techniques on the convex relaxation of the nonconvex vector space problem.
We note that the approach is a covering space approach which relies on the fact that the covering space is a vector space which limits its generalization to general manifolds. In connection with 
$\mathbb S^1$ valued data we also point out the paper \cite{storath2016exact}
which provides an exact solver for the univariate $L^1$-$\TV$ problem \eqref{eq:abstract_min_problem_tv_manifold_univariate} with $q=1$.
 
For general manifolds there are three conceptually different approaches to TV regularization. 
The authors of \cite{LSKC13} reformulate the TV problem as a multi-label optimization problem.
More precisely, they consider a lifted reformulation in a vector-space setting, where the unknown takes values in the space of probability measures on the manifold (rather than the manifold itself), such that it assigns a probability for each given value on the manifold. Constraining the values of the unknown to be delta peaks, this would correspond to an exact reformulation whereas dropping this constraint yields a convex relaxation. After discretization, the unknown takes values in the unit simplex assigning a probability to each element of a discrete set of possible values. 
This corresponds to a lifting of the problem to higher dimensions, where the number of values the unknown is allowed to attain defines the dimensionality of the problem. Having a vector-space structure available, the lifted problem is then solved numerically using duality-based methods. We refer to \cite{LSKC13} for details and to \cite{Vogt2019lifting_mh} for an overview of research in that direction and extensions.

Another approach can be found in the paper \cite{grohs2016total}. 
There, the authors employ an iteratively reweighted least squares (IRLS) algorithm to the isotropic discrete TV functional \eqref{eq:abstract_min_problem__tv_manifold_bivariate}. 
The idea of the IRLS is to replace the distance terms in the TV regularizer by squared distance terms and to introduce a weight for each summand of the regularizer. Then, fixing the weights, the problem is differentiable and can be solved using methods for differentiable functions 
such as a gradient descent scheme. 
In a next step, the weights are updated where a large residual part of a summand results in a small weight, and the process is iterated. This results in an alternating minimization algorithm. 
The authors show convergence in the case of Hadamard spaces and for data living in a half-sphere. 
We mention that IRLS minimization is frequently applied for recovering sparse signals and that it has been also applied to scalar TV minimization in \cite{rodriguez2006iteratively}.
In connection with this, we also mention the paper \cite{steidl16half_quadratic} which considers half-quadratic minimization approaches that are generalizations of \cite{grohs2016total}.

Finally, the approach of \cite{weinmann2014total} to TV regularization       
employs iterative geodesic averaging to implement cyclic and parallel proximal point algorithms. 
The main point here is that the appearing proximal mappings can be analytically computed and the resulting algorithms exclusively perform iterative geodesic averaging. This means that only points on geodesics have to be computed. We will elaborate on this approach in the following.
In connection with this, we also mention the paper \cite{baust2016combined}
where a generalized forward-backward type algorithm is proposed to solve a related problem in the context of DTI; see also \cite{baust2015total, stefanoiu2016joint} in the context of shape spaces. 

The approach of \cite{weinmann2014total} relies on the concepts of cyclic proximal point algorithms (CPPAs) and parallel proximal point algorithms (PPPA) in a manifold.  
A reference for cyclic proximal point algorithms in vector spaces is \cite{Bertsekas2011in}.
In the context of nonlinear spaces, the concept of CPPAs was first proposed in \cite{bavcak2013computing}, where it is employed to compute means and medians in Hadamard spaces.
In the context of variational regularization methods for nonlinear, manifold-valued data,
they were first used in \cite{weinmann2014total}, which also proposed the PPPA in the manifold setting.

\noindent \textbf{CPPAs and PPPAs.} The idea of both CPPAs and PPPAs is to decompose a functional $F: \Mc^N \to \mathbb R$ to be minimized into 
basic atoms $(F_i)_{i=1}^K$ and then to compute the proximal mappings of the atoms $F_i$ iteratively. For a CPPA, this is done in a cyclic way, and for a PPPA, in a parallel way. 
More precisely, assume that 
\begin{equation}\label{eq:Decomp4CPPA}
F = \sum\nolimits_{i=1}^K F_i
\end{equation}
and consider the proximal mappings \cite{moreau1962fonctions, ferreira2002proximal, azagra2005proximal} $\prox_{\lambda F_i}: \mathcal{M}^N \rightarrow \mathcal{M}^N$ given as
\begin{align} \label{eq:prox_mapping_abstract}
\prox_{\lambda F_i} (x) = \argmin_y \, F_i(y) + 
\tfrac{1}{2 \lambda}  \sum\nolimits_{j=1}^N\dist(x_j,y_j)^2.  
\end{align}
One cycle of a CPPA then consists of applying each proximal mapping $\prox_{\lambda F_i}$ once in a prescribed order, e.g., $\prox_{\lambda F_1},$ $\prox_{\lambda F_2},$ $\prox_{\lambda F_3}, \ldots,$ or, generally,
$\prox_{\lambda F_{\sigma(1)}},$ $\prox_{\lambda F_{\sigma(2)}},$ $\prox_{\lambda F_{\sigma(3)}},$ $\ldots,$ 
where the symbol $\sigma$ is employed to denote a permutation.
The cyclic nature is reflected in the fact that the output of  $\prox_{\lambda F_{\sigma(i)}}$ is used as input for $\prox_{\lambda f_{\sigma(i+1)}}.$ 
Since the $i$th update is immediately used for the $(i+1)$th step, it can be seen as a Gauss-Seidel-type scheme. We refer to Algorithm~\ref{alg:cppa} for its implementation in pseudocode.

\begin{algorithm}[t]
  \begin{algorithmic}[1]
\State \textbf{CPPA}($x^0,(\lambda_k)_k,(\sigma(j))_{j=1}^K)$

\State $k=0$, $x_0^0 = x^0$
\State \quad \textbf{repeat} until stopping criterion fulfilled
 \State \qquad \textbf{for}$j=1,\ldots,K$
	\State \quad \qquad $x_{j}^k = \prox_{\lambda_k F_{\sigma(j)}} (x_{j-1}^k) $
\State \qquad $x^{k+1}_0 = x^k _{K},$
 \quad $k\gets k+1$
\State \Return{$x^k_0$}
\end{algorithmic}
\caption{CPPA for solving $\min_x F(x)$ with $F = \sum_{j=1}^K F_j$\label{alg:cppa}}
\end{algorithm}

\begin{algorithm}[t]
  \begin{algorithmic}[1]
\State \textbf{PPPA}($x^0,(\lambda_k)_k)$

\State $k=0$,
\State \quad \textbf{repeat} until stopping criterion fulfilled
 \State \qquad \textbf{for}$j=1,\ldots,K$
	\State \quad \qquad  $x^{k+1}_j = \prox_{\lambda_k F_{j}}(x^{k})$
\State \qquad $x^{k+1} = \mean_j \left(x^{k+1}_j\right),$
 \quad $k\gets k+1$
\State \Return{$x^k$}
\end{algorithmic}
\caption{PPPA for solving $\min_x F(x)$ with $F = \sum_{j=1}^K F_j$\label{alg:pppa}}
\end{algorithm}

A PPPA consists of applying the proximal mapping to each atom $F_i$ to the output of the previous iteration $x^{k}$ in parallel and then averaging the results, see Algorithm~\ref{alg:pppa}. Since it performs the elementary update steps, i.e., the evaluation of the proximal mappings, in parallel it can be seen as update pattern of Jacobi type.
In Algorithm \ref{alg:pppa}, the symbol \emph{mean} denotes the generalization of the arithmetic average to a Riemannian manifold, which is the well known intrinsic mean, i.e., 
given $z_1,\ldots,z_K$ in $\mathcal M,$ a mean $z^\ast \in \mathcal M$ is defined by
(cf. \cite{karcher1977riemannian,kendall1990probability, pennec2006riemannian, fletcher2007riemannian})
\begin{align} \label{eq:DefMean}
z^\ast = \mean_j \left(z_j\right)  = \argmin_{z \in \Mc} \sum\nolimits_{j=1}^K \dist(z,z_j)^2. 
\end{align}
Please note that this definition is employed component-wise for $x^{k+1}$ in Algorithm \ref{alg:pppa}.
We note that, if the $(F_i)_i$ are lower semi continuous, both the minimization problem for the proximal mapping and for the mean admit a solution. On general manifolds, however, the solution is not necessarily unique. For arguments whose points are all contained in a small ball (whose radius depends on the sectional curvature $\Mc$) it is unique, see \cite{ferreira2002proximal, azagra2005proximal,kendall1990probability, karcher1977riemannian} for details. This is a general issue in the context of manifolds that are -- in a certain sense -- a local concept
involving objects that are often only locally well defined. In case of ambiguities, we hence consider the above objects as set-valued quantities.

During the iteration of both CPPA and PPPA, the parameter $\lambda_k$ of the proximal mappings is successively decreased. 
In this way, the penalty for deviation from the previous iterate is successively increased.
It is chosen in a way such that the sequence $(\lambda^k)_k$ is square-summable but not summable.  
Provided that this condition holds, the CPPA can be shown to converge to the optimal solution of the underlying minimization problem, at least in the context of Hadamard manifolds and convex $(F_i)_i$, see \cite[Theorem 3.1]{bacak2014convex}. The same statement holds for the PPPA, see \cite[Theorem 4]{weinmann2014total}. The mean can be computed using a gradient descent or a Newton scheme.
To reduce the computation time further, it has been proposed in \cite{weinmann2014total}
to replace the mean by another construction (known as geodesic analogues in the subdivision context \cite{wallner2005convergence}) which is an approximation of the mean that is computationally less demanding. As above, in the context of Hadamard manifolds and convex $(F_i)_i,$ the convergence towards a global minimizer is guaranteed, see \cite[Theorem 7]{weinmann2014total}.
For details we refer to the above reference.

\begin{figure}[t]
	\def\figfolder{experiments/colorHcl/}
	\def\figwidth{0.24\textwidth}

	\hfill
	\includegraphics[width= \figwidth]{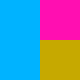}
	\hfill
	\includegraphics[width= \figwidth]{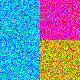}		
	\hfill
	\includegraphics[width= \figwidth]{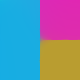}
	\hfill
	\includegraphics[width= \figwidth]{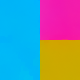}
	\caption{The effect of $\ell^2$-TV denoising in LCh space ($\alpha = \protect\input{\figfolder alpha.txt}$).
		{\em Left.} Ground truth.
		{\em Middle left.}  Noisy input image corrupted by Gaussian noise on each channel. 
		{\em Middle right.} The $\ell^2$-TV reconstruction in linear space.
		{\em Right.}  
		The $\ell^2$-TV reconstruction in the nonlinear LCh color space. %
		Using the distance in the non-flat LCh metric can lead to higher reconstruction quality. 
	}		
	\label{fig:LCh}
\end{figure}

\noindent \textbf{Proximal mappings for the atoms of the TV functions.} 
Now we consider a splitting of the univariate problem  \eqref{eq:abstract_min_problem_tv_manifold_univariate} and the bivariate problem \eqref{eq:abstract_min_problem__tv_manifold_bivariate} into basic atoms such that the CPPA and the PPPA can be applied. Regarding \eqref{eq:abstract_min_problem_tv_manifold_univariate} we use the atoms
\begin{align} \label{eq:tv_manifoldParallelUnivAlgoDecomposition}
F_1(x) := \tfrac{1}{q} \sum_{i=1}^{N} \dist(x_{i},f_i)^q, \quad  F_2(x) = \sum_{ \substack{i=1 \\ i \text{ odd}}}^{N-1} \dist(x_i,x_{i+1}), \quad  F_3(x) = \sum_{ \substack{i=1 \\ i \text{ even}}}^N \dist(x_i,x_{i+1}).
\end{align}
Regarding \eqref{eq:abstract_min_problem__tv_manifold_bivariate}, we consider the case $p=1$ and again define $F_1$ to be the data term, $F_2$ and $F_3$ to be a splitting of the sum $\sum_{i,j} \dist(x_{i,j},x_{i+1,j})$ into even and odd values of $i$ and $F_4$ and $F_5$ to be a splitting of the sum $\sum_{i,j} \dist(x_{i,j},x_{i,j+1})$ into even and odd values of $j$. 
With these splittings, all summands in the atom $(F_i)_i$ decouple such that the computation of the proximal mappings reduces to a point-wise computation of the proximal mappings of 
\begin{align}\label{eq:elemRed}
 x \mapsto g_1(x,f):= \frac{1}{q}\dist(x,f)^q \quad \text{ and } \quad (x_1,x_2)   \mapsto g_2(x_1,x_2) = \dist(x_1,x_2) .
\end{align}
From the splitting \eqref{eq:tv_manifoldParallelUnivAlgoDecomposition} 
(and its bivariate analogue below \eqref{eq:tv_manifoldParallelUnivAlgoDecomposition}) 
together with \eqref{eq:elemRed}
we see that within a PPPA all proximal mappings of the basic building blocks $g_1,g_2$ can be computed in parallel and the computation of each mean only involves $3$ points in the manifold $\mathcal M$ in the univariate setting and $5$ points in the multivariate setting.
For a CPPA we see that a cycle has length $3$ and $5$ in the univariate and bivariate situation, respectively, and that within each atom $F_i$ the proximal mappings of the respective terms of the form $g_1,g_2$ can be computed in parallel.

For the data term, the proximal mappings $\prox_{\lambda g_1}$ are explicit for $q=1$ and $q=2$ and, as derived in \cite{ferreira2002proximal}, are given as
\begin{align}\label{eq:tv_manifoldProxData}
(\prox_{\lambda g_1(\cdot ,f)})_{j}(x) =  [x,f]_t
\end{align}
\text{ where } 
\begin{align}\label{eq:t4data}
 t= \tfrac{\lambda}{1+\lambda} \ \text{ for } \ q=2, \qquad  
 t= \min \left(\tfrac \lambda {\dist(x,f)},1 \right) \ \text{  for } \ q=1.
\end{align}
Here, we use the symbol $[\cdot,\cdot]_t$ to denote the point reached after time $t$
on the (non unit speed) length-minimizing geodesic starting at the first argument reaching the second argument at time $1$.
(Note, that up to sets of measure zero, length minimizing geodesics are unique, and in the extraordinary case of non-uniqueness we pick may one of them.) 

Regarding $g_2$, it is shown in \cite{weinmann2014total} that the proximal mappings are given in closed form as
\begin{align}\label{eq:tv_manifoldProxReg}
\prox_{\lambda g_2 } ((x_1,x_2))=  ([x_1,x_2]_t, [x_2,x_1]_t), \quad \text{where }
	t = \min \left(\frac{\lambda \alpha} {\dist(x_{1},x_{2})},\frac{1}{2} \right).
\end{align}
Here, for each point, the result is a point on the geodesic segment connecting two arguments.

It is important to note that the point $p_t = [p_0,p_1]_t$ on the geodesic connecting two points 
	$p_0,p_1$ is given in terms of the Riemannian exponential map $\exp$ and its inverse denoted by $\log$ or $\exp^{-1}$ by
	\begin{align}\label{eq:FeodesicsInTermsOfLogExpHolWei}
	p_t = [p_0,p_1]_t = \exp_{p_0}(t \ \log_{p_0}p_1).
	\end{align}
	Here, $v:=\log_{p_0}p_1$ denotes that tangent vector sitting in $p_0$ such that $\exp_{p_0}v = p_1.$
	The tangent vector $v$ is scaled by $t,$ and then the application of the $\exp$-map yields $p_t.$
	More precisely, $\exp_{p_0}$ assigns the point $p_t=\exp_{p_0} tv $ to the tangent vector $tv$ by evaluating the geodesic starting in $p_0$ with tangent vector $tv$ at time $1$.
\begin{figure}[t]
	\def\figfolder{./experiments/S2/}
	\def\figwidth{0.32\columnwidth}
	\includegraphics[width=\figwidth]{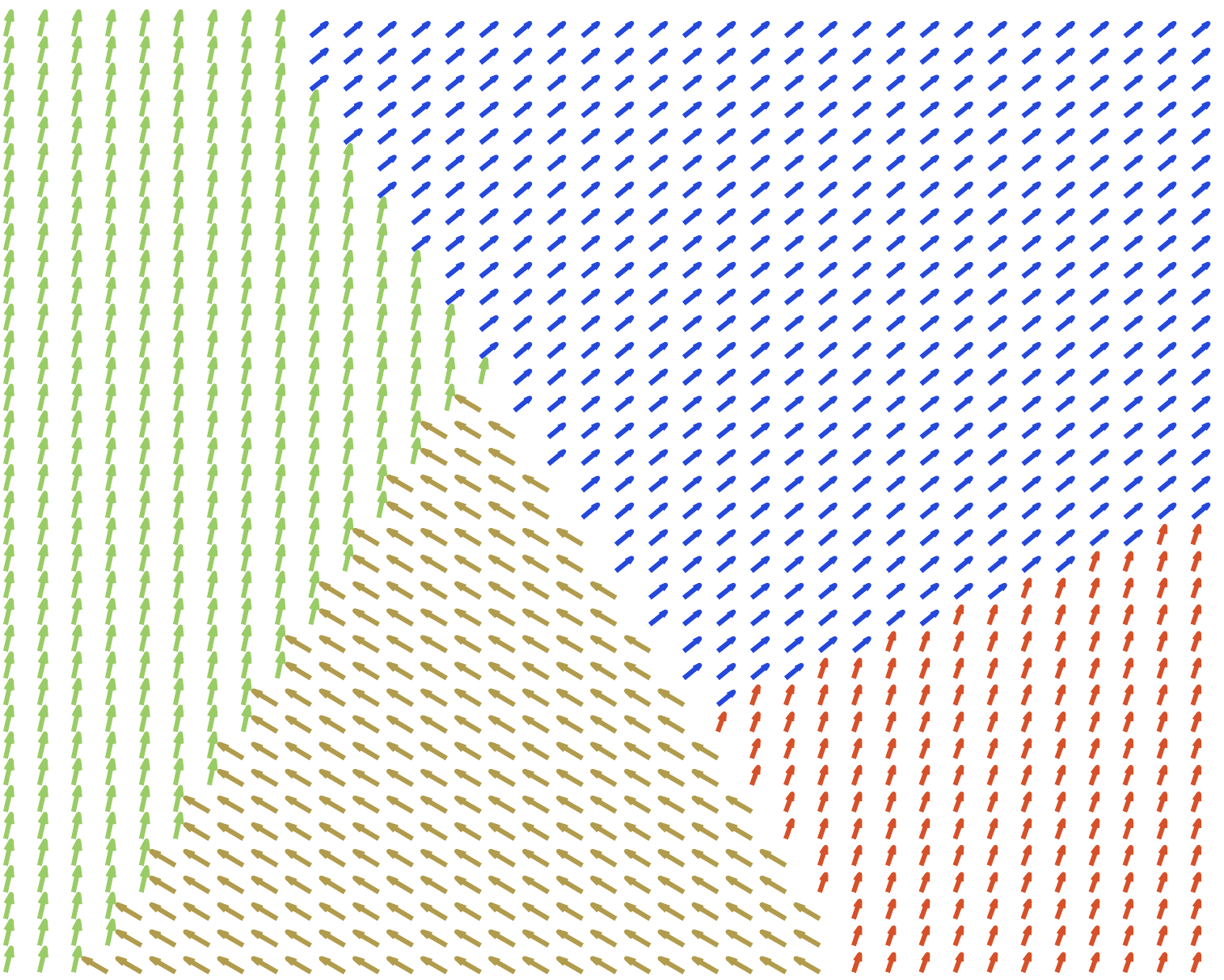}
	\hfill
	\includegraphics[width=0.33\columnwidth]{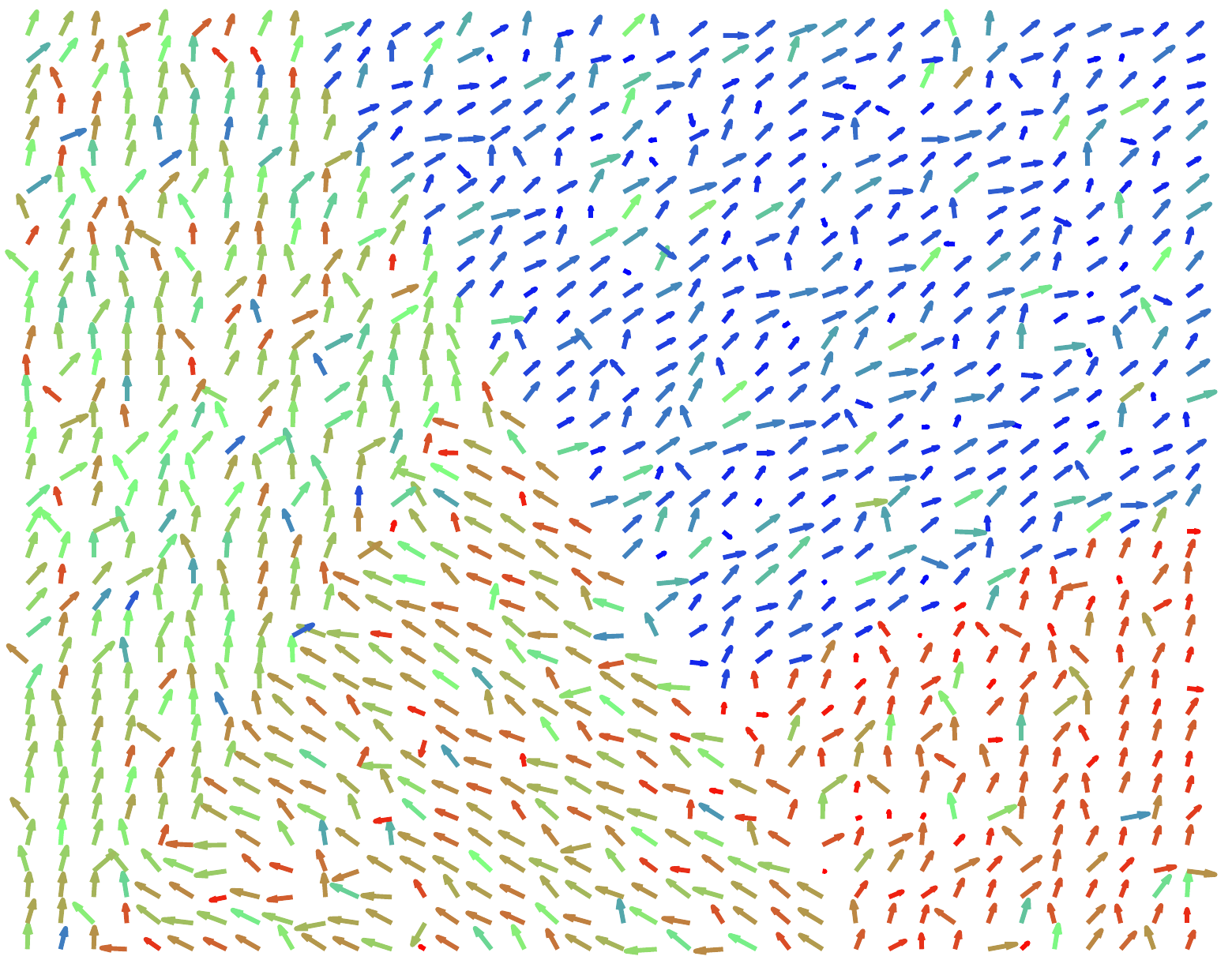}
	\hfill
	\includegraphics[width=\figwidth]{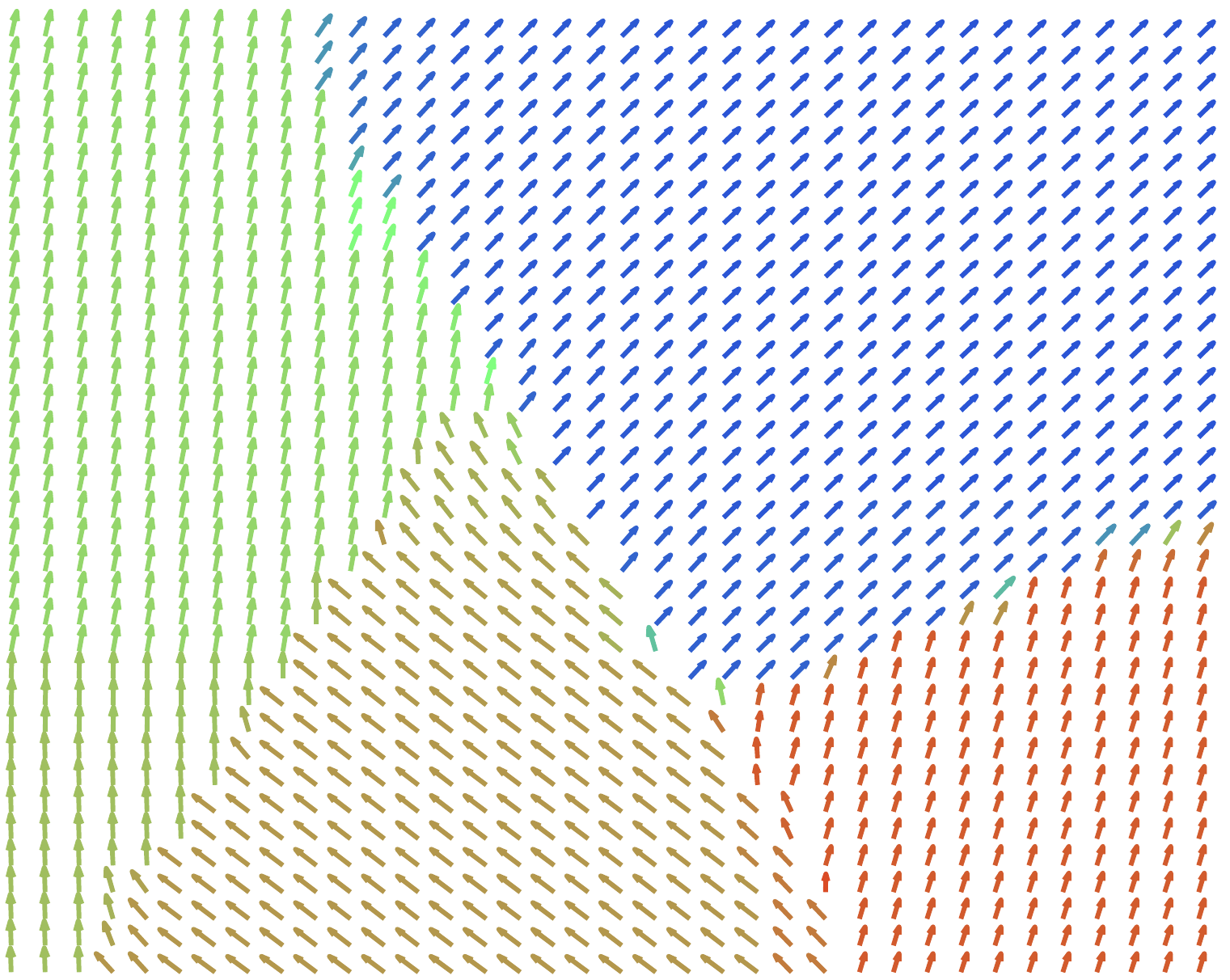}
	\caption{
		Denoising of an $\mathbb S^2$-valued image. The polar angle is coded both as length of the vectors and as color (red pointing towards the reader, blue away from the reader).
		\emph{Left.} Synthetic image.
		\emph{Center.} Noisy data (corrupted by von Mises-Fisher noise of level $\kappa =\protect\input{\figfolder kappa.txt}$).
		\emph{Right.} $\ell^{\protect\input{\figfolder p.txt}}$-TV regularization
		using $\alpha = \protect\input{\figfolder alpha.txt}.$ The noise is almost completely removed whereas the jumps are preserved.
	}	
	\label{fig:s2_exp}			
\end{figure}

We note that also the proximal mappings of the classical Tichanov regularizers as well as of the Huber regularizers mentioned above have a closed form representation in terms of geodesic averaging as well. Further, there are strategies to approximate intrinsic means by iterated geodesic averages to speed up the corresponding computations.
For details on these comments we refer to \cite{weinmann2014total}.

Plugging in the splittings and proximal mappings as above into the Algorithms \ref{alg:cppa} and \ref{alg:pppa} yields a concrete implementation for the TV-regularized denoising of manifold-valued data. Regarding convergence, we have the following result.

\begin{theorem} \label{thm:ConvergenceAlgA}
	For data in a (locally compact) Hadamard space and a parameter sequence $(\lambda_k)_k$ which is squared summable but not summable, the iterative geodesic averaging algorithms for TV-reglarized denosing (based on the CPPA, the PPPA, as well as the inexact approximative and fast variant of the PPPA) converge towards a minimizer of the $\ell^p$-$\TV$ functional. 	
\end{theorem}

We further remark that the statement remains true when using the Huber potential mentioned above either as data term or for the regularization, as well as when using quadratic variation instead of $\TV.$ A proof of this statement and more details on the remarks can be found in \cite{weinmann2014total}.
\begin{figure}[t]
	\centering
	\def\figfolder{./experiments/SO3/}
	\includegraphics[width=0.32\textwidth]{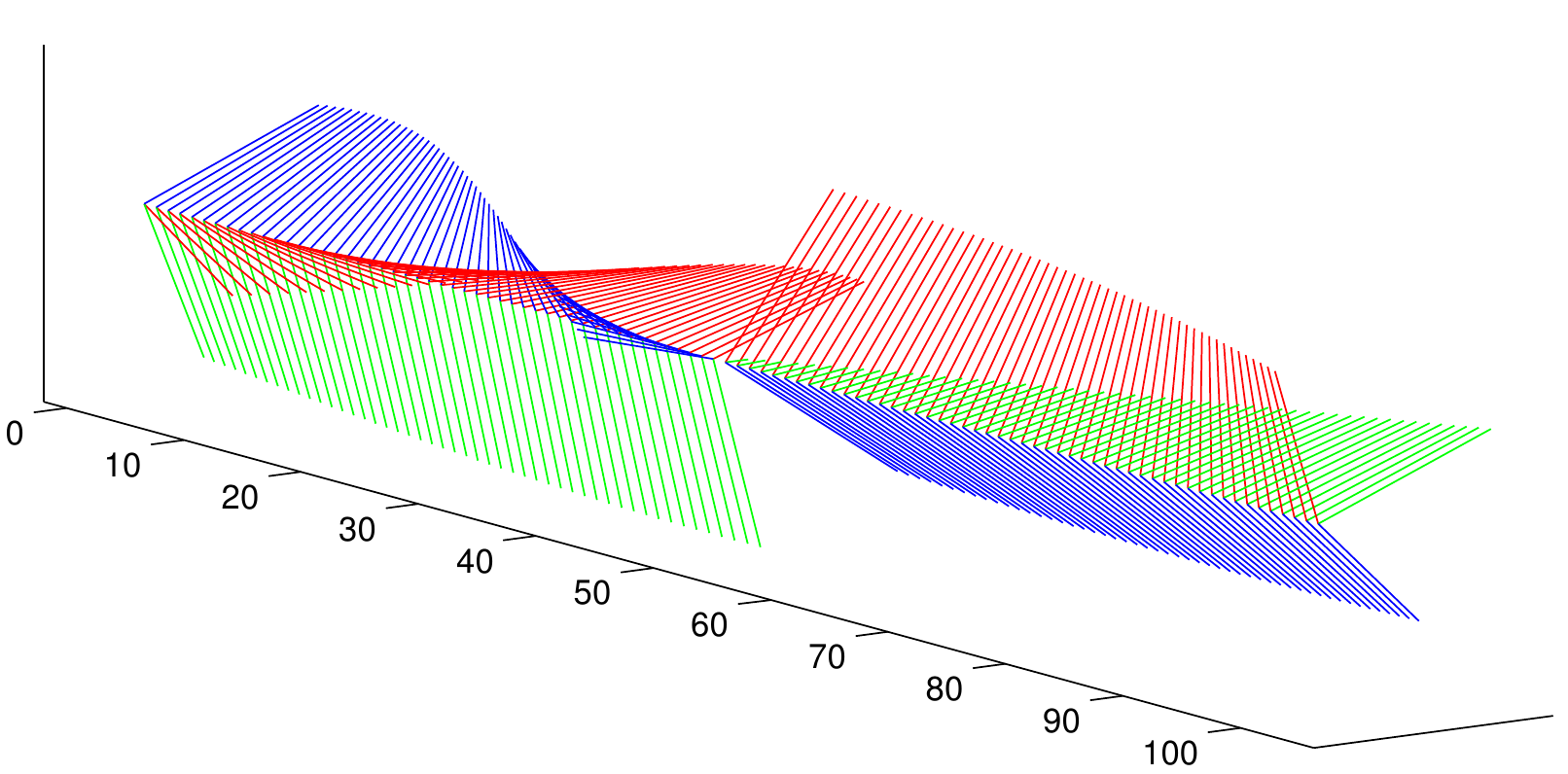}
	\includegraphics[width=.32\textwidth]{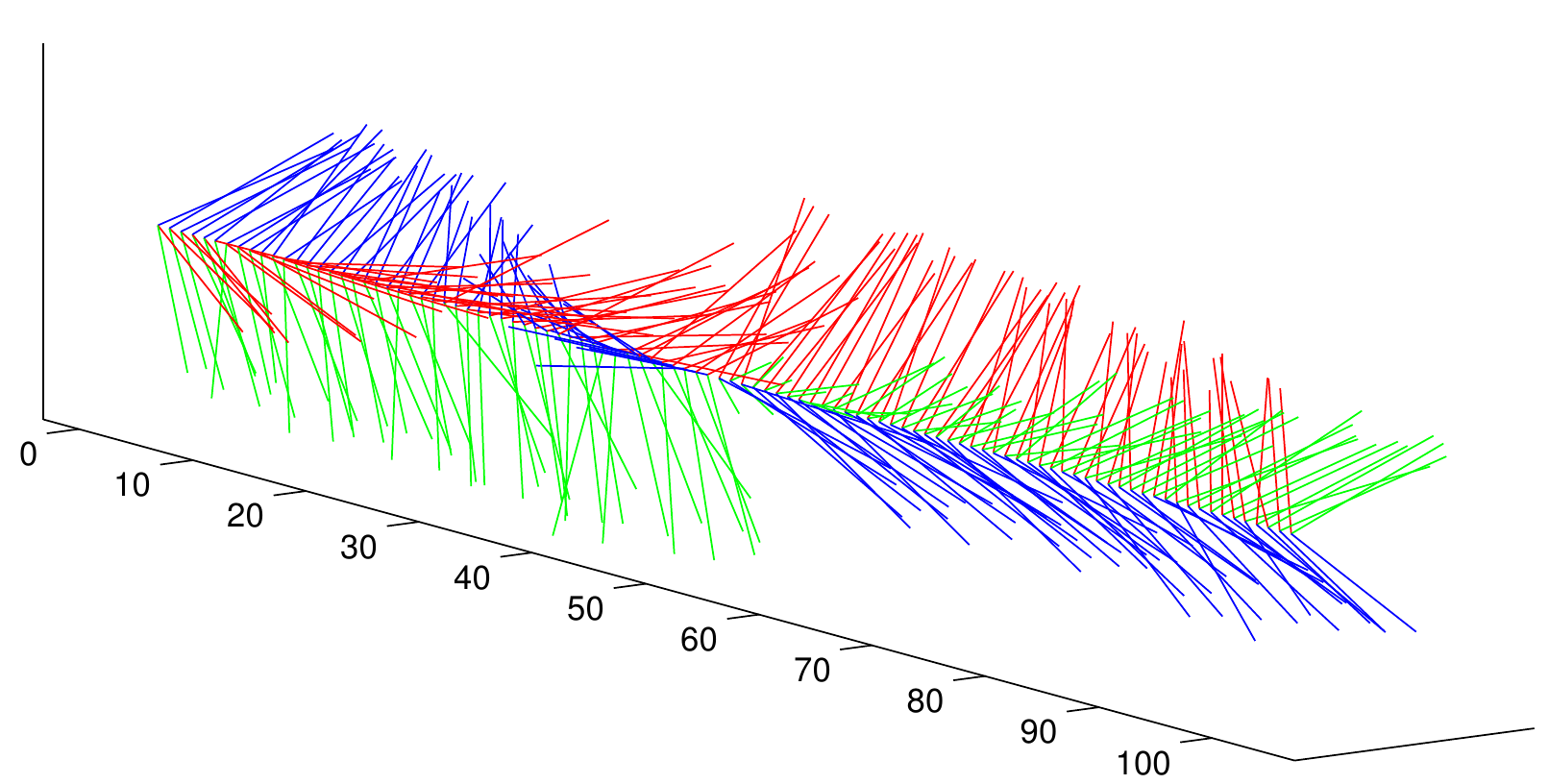}
	\includegraphics[width=.32\textwidth]{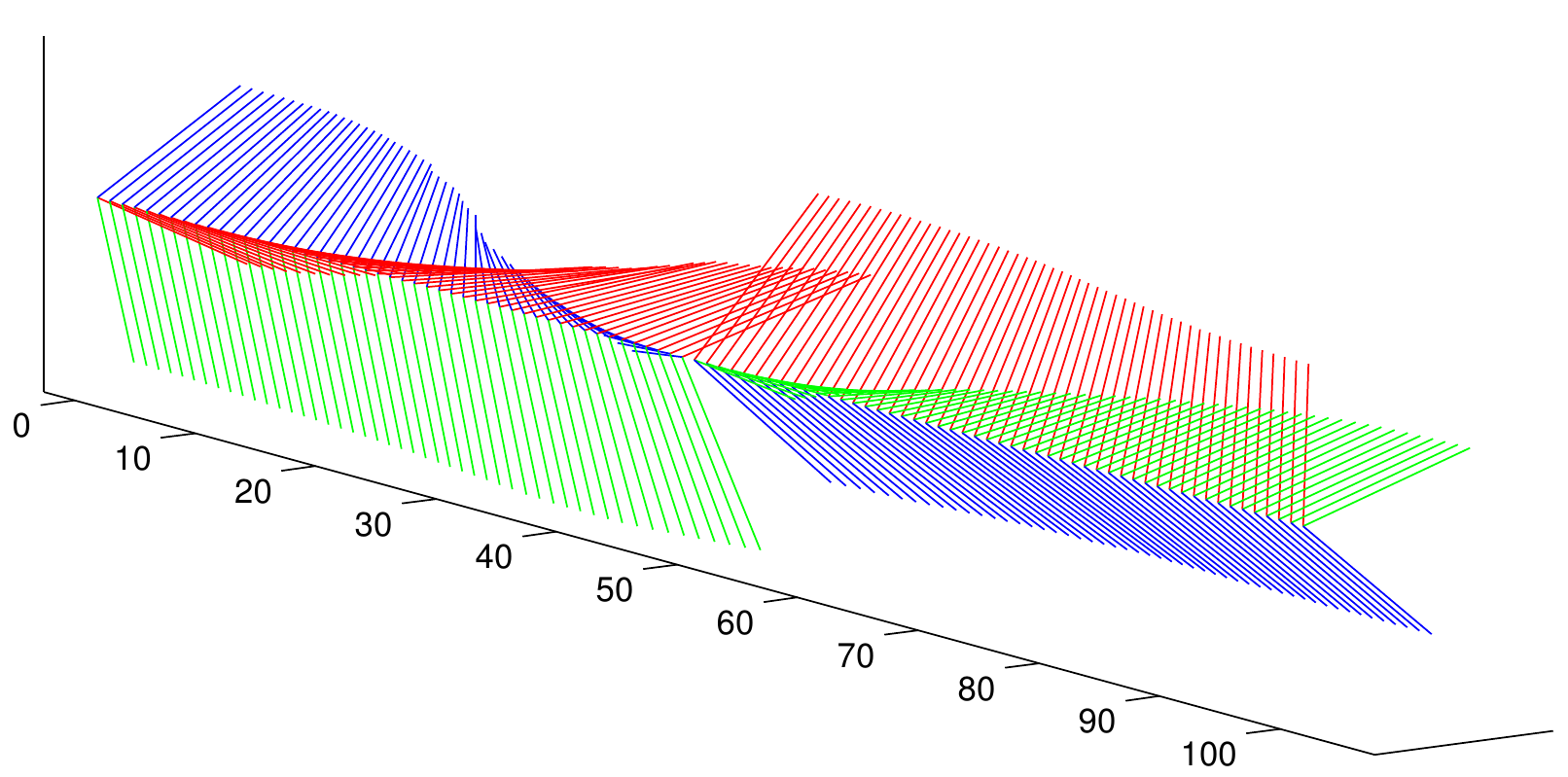}

	\caption{Result ({\em right}) of denoising an $\mathrm{SO}(3)$-valued noisy time-series ({\em center}) using the inexact parallel algorithm for 
		$L^2$-$\TV$ regularization with	$\alpha = \protect\input{\figfolder alpha.txt}.$
		({\em Left:} Ground truth.)
		Here, an element of $\mathrm{SO}(3)$ is visualized by the rotation of a tripod.
		We observe that the noise is removed and the jump is preserved. 	
	}
	\label{fig:so3}
\end{figure}

We illustrate the algorithms with some examples.
First we consider denoising in the LCh color space. As explained above, the underlying manifold is $\mathbb S^1 \times \mathbb R^2.$
The exponential and its inverse are given componentwise by the respective mappings on
$\mathbb R^2$ and $\mathbb S^1$. 
By \eqref{eq:FeodesicsInTermsOfLogExpHolWei}, this allows to compute the involved proximal mappings
via \eqref{eq:tv_manifoldProxData}, \eqref{eq:t4data} and \eqref{eq:tv_manifoldProxReg}, respectively.
We point out that in spite of the separability of the exponetial and its inverse, the proposed algorithm 
is in general not equivalent to performing the algorithm on $\mathbb R^2$ and $\mathbb S^1$ separately. 
The reason is that the parameter $t$ in  \eqref{eq:t4data} and \eqref{eq:tv_manifoldProxReg}
depend nonlinearly on the distance in the product manifold (except for $p,q = 2$).
In Figure~\ref{fig:LCh} we illustrate the denoising potential of the proposed scheme in the LCh space.
Here, the vector-space computation was realized using the split Bregman method for vectorial TV regularization \cite{goldstein2009split, getreuer2012rudin} and we optimized the parameters of both methods with respect to the peak signal to noise ratio.

As a second example we cosider noisy data on the unit sphere $\mathbb S^2$ (in $\mathbb R^3$) .
In Figure~\ref{fig:s2_exp}, we test the denoising potential of our algorithm on a noisy (synthetic) spherical-valued image.
As noise model on $\mathbb S^2,$ we use the von Mises-Fisher distribution having the probability density 
$
h(x) = c(\kappa) \exp (\kappa x \cdot \mu).
$
Here, $\kappa >0$ expresses the concentration around the mean orientation $\mu \in \mathbb S^2$
where a higher $\kappa$ indicates a higher concentration of the distribution and 
$c(\kappa)$ is a normalization constant. 
We observe in Figure~\ref{fig:s2_exp} that the noise is almost completely removed by TV minimization and that the edges are retained. 

In Figure~\ref{fig:so3} we consider an univariate signal with values in the
special orthogonal group $\mathrm{SO}(3)$ consisting of all orthogonal $3\times 3$ matrices with determinant one. We see that the proposed algorithm removes the noise and that the jump is preserved.
Finally, we consider real InSAR data \cite{massonnet1998radar,rocca1997overview} in Figure~\ref{fig:SAR}.
InSAR images consist of phase values such that the underlying manifold is the one-dimensional sphere $\mathbb S^1$.
The image is taken from \cite{rocca1997overview}. 
We apply total variation denoising using $\ell^2$ and $\ell^1$ data terms. 
We observe that TV regularization reduces the noise significantly.
The $\ell^1$ data term seems to be more robust to outliers than the $\ell^2$ data term.

\begin{figure}[t]
	\def\figfolder{./experiments/SAR/}
	\hfill	
	\includegraphics[width= .32\textwidth]{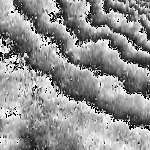}
	\hfill
	\includegraphics[width= .32\textwidth]{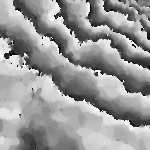}
	\hfill
	\includegraphics[width= .32\textwidth]{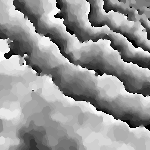}
	\hfill
	\caption{
		Total variation denoising of an $\mathbb{S}^1$-valued InSAR image (real data, {\em left}) using
		$L^2$-$\TV$ regularization ($\alpha$=\protect\input{\figfolder alpha_2.txt}, {\em middle}) and 
		$L^1$-$\TV$ regularization ($\alpha$=\protect\input{\figfolder alpha_1.txt}, {\em right}). 
		Here, the circle $\mathbb{S}^1$ is represented as an interval with endpoints identified, i.e., white and black represent points nearby.		
		Total variation minimization reliably removes the noise while preserving the structure of the image. 
	}		
	\label{fig:SAR}
\end{figure}

\section{Higher Order Total Variation Approaches, Total Generalized Variation}
\label{sec:higherOandTGV}

It is well known in the vector space situation (and analytically confirmed for instance in \cite{Carioni18sparsity_mh,Boyer19representer_mh}) that TV regularization has a tendency to produce piecewise-constant results, leading to artificial jump discontinuities in case of ground truth data with smooth regions. Classical $H^1$ regularization avoids this effect. However, $H^1$ regularity does not allow for jump discontinuities, which can be seen as motivation for considering non-smooth higher order approaches. While second order TV regularization \cite{demengel1984bounded_heasian_mh,hinterberger2006boundedhessian_mh}, i.e., penalizing the Radon norm of the second order distributional derivative of a function, is a first attempt in this direction, one can show that functions whose second order distributional derivative can be represented by a Radon measure again cannot have jumps along (smooth) hypersurfaces \cite{holler14inversetgv_mh}. This disadvantage is no longer present when using a combination of first and second order TV via infimal convolution, i.e., 
\[ \IC_\alpha(u) = \inf_v \alpha_1\TV(u-v) + \alpha_0\TV^2(v),\] as originally proposed in \cite{chambolle1997inftv_mh}. Here, $\alpha=(\alpha_1,\alpha_0) \in (0,\infty) ^2$ are two weights. Regularization with $\TV$-$\TV^2$ infimal convolution finds an optimal additive decomposition of the unknown $u$ in two components, where one yields minimal cost for TV and the other one for second order TV. Extending on that, the (second order) total generalized variation (TGV) functional \cite{bredies2010tgv_mh} optimally balances between first and second order derivatives on the level of the gradient rather than the function, i.e., is given as
\[ \TGVat(u) = \inf _w \alpha_1\|\nabla u - w \|_\Mc + \alpha_0\|\symgrad w \|_\Mc,
\]
where $\symgrad w = 1/2(Jw + Jw^T)$ is a symmetrization of the Jacobian matrix field $Jw$ and again $\alpha=(\alpha_1,\alpha_0) \in (0,\infty) ^2$ are two weights.
This provides a more flexible balancing between different orders of differentiation and, in particular in situations when an optimal decomposition on the image level is not possible, further reduces piecewise constancy artifacts still present with $\TV$-$\TV^2$ infimal convolution, see \cite{bredies2010tgv_mh}.

Motivated by the developments for vector spaces, and due to the challenges appearing when extending them to manifold-valued data, several works deal with developing non-smooth higher order regularization in this setting. In the following we motivate and review some existing approaches and strive to present them in a common framework. Following the existing literature, we consider a discrete setting. 

\subsection{Models}
\label{sec:higherOandTGVmodel}
First, we define the above-mentioned higher order regularization functionals in a discrete, vector-space setting. To this aim, for $u = (u_{i,j})_{i,j} \in \RR^{N\times M}$, let 
$\delta_x:\RR^{N \times M} \rightarrow \RR^{(N-1) \times M}$, $(\delta_xu)_{i,j} = u_{i+1,j} - u_{i,j}$ and 
$\delta_y:\RR^{N \times M} \rightarrow \RR^{N \times (M-1) }$, $(\delta_y u)_{i,j} = u_{i,j+1} - u_{i,j}$
be finite differences (on a staggered grid to avoid boundary effects) w.r.t. the first- and second component, respectively. A discrete gradient, Jacobian and symmetrizied Jacobian are then given as
\begin{align*}
&\nabla u  = (\delta_x u,\delta_yu) ,\qquad  \J(v^1,v^2) = (\delta_xv^1,\delta_yv^2,\delta_y v^1,\delta_x v^2), \\
&\symgrad(v^1,v^2)  = (\delta_xv^1,\delta_yv^2, \tfrac{\delta_y v^1 + \delta_x v^2}{2}),
\end{align*}
respectively, where $v = (v^1,v^2) \in \RR^{(N-1) \times M} \times \RR^{N \times (M-1)}.$ 
We note that different components of $\nabla u,$ $ \J v,$ $ \symgrad v$ have different length. 
Using these objects, we define discrete versions of $\TV$, second order $\TV$, of $\TV$-$\TV^2$ infimal convolution and of $\TGV$ as %
\begin{equation} \label{eq:DefTV2ICTGVFunctionalsVec}
\begin{aligned}
\TV(u) &= \|\nabla u\|_1, \quad \TV^2(u)  = \| \J \nabla u\|_1, \quad   \IC_\alpha(u)   = \min _v \alpha_1\TV(u-v) + \alpha_0 \TV^2(v),\\
\TGVat(u)  & = \min_w \, \alpha_1 \|\nabla u - w \|_1 + \alpha_0 \|\symgrad w \|_1. 
\end{aligned}
\end{equation}
Here, $\|\cdot \|_1$ denotes the $\ell^1$ norm w.r.t. the spatial component and we take an $\ell^p$ norm (with $p \in [1,\infty)$) in the vector components without explicit mentioning, e.g., \newline $\| \nabla u \|_1 := \sum_{i,j} \left( (\delta _x u)_{i,j} ^p + (\delta _y u)^p_{i,j}\right) ^{1/p}$, where we again replace summands containing out-of-bound indices $0$. Note that the most interesting cases are $p=1$ due to advantages for the numerical realization and $p=2$ since this corresponds to isotropic functionals in the infinite-dimensional vector-space case, see for instance \cite{bredies2010tgv_mh} for TGV. Also note that $(\J\nabla u)_{i,j}$ is symmetric, that the symmetric component of $\symgrad w$ is stored only once and that we define
$\| \symgrad w \|_1 := \sum_{i,j} \left( (\delta _x w^1)_{i,j} ^p + (\delta _y w^2)^p_{i,j} + 2 (\frac{\delta _y w^1 + \delta_x w ^2}{2})^p_{i,j} \right) ^{1/p}$ to compensate for that.

Now we extend these regularizers to arguments $u \in \Mc^{N\times M}$
with $\Mc$ being a complete, finite dimensional Riemannian manifold with induced distance
 $\dist.$ 
 For the sake of highlighting the main ideas first, we start with the univariate situation $u = (u_i)_i \in \Mc^{N}$.

Regarding second order $\TV$, following \cite{Bavcak2016tv2_manifold_mh}, we observe (with $\delta$ the univariate version of $\delta_x$) that, for $u \in (\RR^d)^N$ and a norm  $\|\cdot \|$ on $\RR^d$,
\[ \|(\delta \delta u)_i\| = \|u_{i+1} - 2 u_i + u_{i-1}\| = 2 \|\frac{u_{i+1} + u_{i-1}}{2} - u_i\|,\]
where the last expression only requires averaging and a distance measure, both of which is available on Riemannian manifolds. Thus, a generalization of $\TV^2$ for $u \in \Mc^N$ can be given, for $u=(u_i)_i,$ by
\[ \TV^2(u) =  \sum\nolimits_i \dc (u_{i-1},u_i,u_{i+1})\quad \text{where } \dc ( u_-,u_\circ,u_+) = \inf _{c \in [u_-,u_+]_{\frac{1}{2}}} 2\dist(c,u_\circ). \]
Here, $\dc$ essentially measures the distance between the central data point $u_\circ$ and the geodesic midpoint of its neighbors, if this midpoint is unique, and the infimum w.r.t.\ all midpoints otherwise.
Similarly, we observe for mixed derivatives and $u \in (\RR^d)^{N \times N}$ that
\begin{align*}
 \|(\delta_y \delta_x u)\| %
   &= 2 \left\| \tfrac{u_{i+1,j} + u_{i,j-1} }{2} - \tfrac{u_{i,j}+u_{i+1,j-1}}{2} \right\|.
\end{align*}
An analogue for $u \in \Mc^{M\times M}$ is hence given by 
\[ \dcc (u_{i,j},u_{i+1,j},u_{i,j-1},u_{i+1,j-1}) = \inf_{\substack{
c_1 \in [u_{i+1,j},u_{i,j-1}]_{\frac{1}{2}}, 
 c_2 \in [u_{i,j},u_{i+1,j-1}]_{\frac{1}{2}}
 }
} 2\dist(c_1,c_2),
\]
and similarly for $\delta_x \delta_y$. Exploiting symmetry, we only incorporate $(\delta_y\delta _x u)$ and define
\begin{align*}
\TV^2(u) =& \sum_{i,j} \Big(  \dc(u_{i-1,j},u_{i,j},u_{i+1,j})^p + \dc(u_{i,j-1},u_{i,j},u_{i,j+1})^p \\
& + 2\dcc (u_{i,j},u_{i+1,j},u_{i,j-1},u_{i+1,j-1})^p  \Big)^{1/p}.
\end{align*}
This generalizes second order $\TV$ for manifold-valued data while still relying only on point-operations on the manifold. As will be shown in Section  \ref{sec:higherOandTGValgo}, numerical result for $\TV^2$ denoising show less staircasing  than $\TV$ denoising. However, it tends towards oversmoothing 
which is expected from the underlying theory and corresponding numerical results in the vector space case.

A possible extension of $\TV$-$\TV^2$ infimal-convolution to manifolds
is based on a representation in the linear space case given as
\[ \IC_\alpha(u) = \inf_v\, \alpha_1\TV(u- v ) + \alpha_0\TV^2(v) = (1/2)\inf_{v,w: u =\frac{v+w}{2}} \alpha_1\TV(v) + \alpha_0\TV^2(w) ,\]
where $u = (u_{i,j})_{i,j} \in (\RR^d)^{N \times M}.$ 
This representation was taken in \cite{Steidl17_infcon_manifold_mh,bergmann2018ictv_tgv_mh} and extended for $u = (u_{i,j})_{i,j} \in \Mc^{N \times M}$ (up to constants) via
\[ \IC (  u ) =   \tfrac{1}{2} \ \inf_{v,w} \ \alpha_1 \TV(v) +\alpha_0 \TV^2 (w) \quad \text{s.t. } u_{i,j} \in [[v_{i,j},w_{i,j}]]_{\frac{1}{2}},\]
where $v= (v_{i,j})_{i,j}$ and $w= (w_{i,j})_{i,j}.$  
Following \cite{Steidl17_infcon_manifold_mh}, we here use the symbol $[[v_{i,j},w_{i,j}]]$ instead of $[v_{i,j},w_{i,j}]$, where we define the former to include also non-distance minimizing geodesics. 

In order to generalize the second order TGV functional to a manifold setting, we consider \eqref{eq:DefTV2ICTGVFunctionalsVec} for vector spaces.
This definition (via optimal balancing) requires to measure the distance of (discrete) vector fields that are in general defined in different tangent spaces. One means to do so is to employ parallel transport for vector fields in order to shift different vector fields to the same tangent space and to measure the distance there.
(We note that the particular locations the vectors are shifted to is irrelevant since the values are equal.)
This approach requires to incorporate more advanced concepts on manifolds. 
Another possibility is to consider a \emph{discrete tangent space} of point tuples via the identification of $v = \log_a(b)$ as a point tuple $[a,b]$ (where $\log$ is the inverse exponential map), and to define a distance-type function on such point tuples. 
Indeed, the above identification is one-to-one except for points on the cut locus (which is a set of measure zero \cite{itoh1998cut_locus_dimension_mh}) and allows to identify discrete derivatives $(\delta_x u)_i = (u_{i+1} - u_{i}) = \log_{u_i} (u_{i+1}) $ as tupel $[u_i,u_{i+1}]$.
Choosing appropriate distance type functions, this identification allows to work exclusively on the level of point-operations and
one might say that the ``level of complexity'' of the latter approach is comparable with that of $\TV^2$ and $\IC.$  Furthermore, a version of the above parallel transport variant can be realized in the tupel setting as well (still incorporating more advanced concepts).
This approach was proposed in \cite{holler18tgvm_mh}; more precisely, an axiomatic approach is pursued in \cite{holler18tgvm_mh}  and  realizations 
via Schild's ladder (requiring only point operations) and parallel transport are proposed and shown to be particular instances of the axiomatic approach.

We explain the approach in more detail, where we focus on the univariate situation first.
We assume for the moment that $D: \Mc ^2 \times \Mc^2$ is an appropriate distance-type function for point tuples. Then,  a definition of $\TGVat$ for an univariate signal $u = (u_i)_i \in \Mc ^N$ can be given as
\[ \TGV_\alpha^2( (u_i)_i) = \inf_{(y_i)_i} \sum_i \alpha_1 D([u_i,u_{i+1}],[u_i,y_i]) + \alpha_0 D([u_i,y_i],[u_{i-1},y_{i-1}]). 
\]
Thus, one is left to determine a suitable choice of $D$. 
One possible choice is based on the  Schild's ladder \cite{kheyfets2000schild_mh} approximation of parallel transport, which is defined as follows (see Figure \ref{fig:schilds_visualization}): Assuming, for the moment, uniqueness of geodesics, define $c = [v,x]_{\frac{1}{2}}$ and $y' = [u,c]_2$. Then $[x,y']$ can be regarded as approximation of the parallel transport of $w = \log_u(v)$ to $x$, which is exact in the vector-space case. Motivated by this, the distance of the tuples $[u,v]$ and $[x,y]$ can be defined as $\dist(y,y').$
 Incorporating non-uniqueness by minimizing over all possible points in this construction to capture also points on the cut locus, yields a distance-type function for point tuples given as
\[ \ds ( [x,y],[u,v]) = \inf_{y' \in \Mc} \dist(y,y') \quad \text{s.t. } y' \in [u,c]_2 \text{ with } c \in [x,v]_{\frac{1}{2}}.
 \]
 
\begin{figure}[t]
\footnotesize
\center
\begin{tikzpicture}[scale=0.2]

\def\cw{8pt}

\filldraw (0,0) circle (\cw) node[left]{$u$};
\filldraw[blue] (20,0) circle (\cw) node[right,blue]{$v$};
\filldraw (0,10) circle (\cw) node[left]{$x$};
\filldraw[blue] (20,13) circle (\cw) node[right,blue]{$y$};

\draw[blue,dotted] (0,0) .. controls (4,-1) and (8,-2) .. (20,0);
\draw[blue,dotted] (0,10) .. controls (4,12) and (8,14) .. (20,13);
\draw (0,0) .. controls (-2,3) and (-2,5) .. (0,10);

\filldraw[darkgreen] (8.5,3.5) circle (\cw) node[darkgreen,above right]{$c$};
\filldraw[darkgreen] (19,8) circle (\cw) node[darkgreen,above right]{$y'$};
\draw[dotted,darkgreen] (0,10) .. controls (4,6) and (8,2) .. (20,0);
\draw[dotted, darkgreen] (0,0) .. controls (6,4) and (8,2) .. (19,8);

\draw[red,->] (0,0) -- (8,-2) node[below] {$w= \log_u(v)$};
\draw[red,->] (0,10) -- (7,8.5) node[above right=0cm] {$\log_x(y') \approx \pt_x(w)$};

\draw[darkgreen,dotted] (0,10) .. controls (5,9) and (7,8) .. (19,8);

\end{tikzpicture}
\caption{Approximate parallel transport of $\log_u(v)$ to $x$ via the Schild's ladder construction. Figure taken from \cite{holler18tgvm_mh}. \label{fig:schilds_visualization}}
\end{figure}
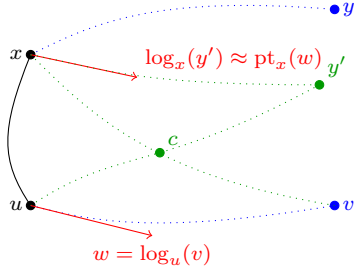

In the particular case that both tuples have the same base point, i.e., $x=u$, it is easy to see that, except in the case of non-unique length-minimizing geodesics, $\ds([x,y],[x,v]) = \dist(v,y)$ such that we can use this as simplification and arrive at a concrete form of manifold-TGV for univariate signals $(u_i)_i$ in $\Mc$ given as 
\[ \STGVat( (u_i)_i) = \inf_{(y_i)_i} \sum_i \alpha_1 \dist(u_{i+1},y_i) + \alpha_0 \ds ([u_i,y_i],[u_{i-1},y_{i-1}]). 
\]
We note that this operation only requires to carry out averaging and reflection followed by applying the distance in the manifold. Thus it is on the same ``level of difficulty'' as $\TV^2$ or $\IC$. For the bivariate situation, the situation if more challenging due to an additional averaging involved in the evaluation of $\symgrad w$. In fact, as described in \cite{holler18tgvm_mh}, there are different possibilities (of varying complexity) to generalize this to the manifold valued setting but there is a unique, rather simple one which in addition transfers fundamental properties of TGV, such as a precise knowledge on its kernel, to the manifold setting. This leads to the definition of $\ds^\sym: (\Mc^2)^4$ which realizes the symmetrized part of $\symgrad w$ in the definition of $\TGV$ and for which, for the sake of brevity, we refer to \cite[Equation 20]{holler18tgvm_mh}. Using $\ds^\sym$, a bivariate version of $\TGV$ for $u = (u_{i,j})_{i,j} \in \Mc^{N \times M}$ is given as
\begin{align} 
&\text{S-TGV}_\alpha^ 2(u)  
 = \min _{y^ 1_{i,j},y^ 2_{i,j}}  \alpha_1 \sum\nolimits _{i,j}  \Big(   d(u_{i+1,j},y^ 1_{i,j})^p + d(u_{i,j+1},y^ 2_{i,j})^p \Big)^{1/p} \notag\\
 & \quad +\alpha _0 \sum\nolimits_{i,j}  \Big( \ds\big([u_{i,j},y^1_{i,j}],[u_{i-1,j},y^1_{i-1,j}]\big)^p + \ds\big([u_{i,j},y^2_{i,j}],[u_{i,j-1},y^2_{i,j-1}]\big)^p  \notag\\
 & \quad + 2^{1-p}\ds^\sym ([u_{i,j},y^1_{i,j}],[u_{i,j},y^2_{i,j}],[u_{i,j-1},y^1_{i,j-1}],[u_{i-1,j},y^2_{i-1,j}])^p \Big)^{1/p}.
\label{eq:tgv_bivariate_schilds}
\end{align}

Naturally, the above definition of $\STGV$ based on the Schild's ladder construction is not the only possibility to extend second order TGV to the manifold setting. As already pointed out, in \cite{holler18tgvm_mh} this was accounted for by an axiomatic approach which, for a suitable generalization, also requires additional properties such as a good description of their kernel, and we will see below that indeed this is possible for $\STGV$. An alternative definition based on parallel transport presented in \cite{holler18tgvm_mh} uses, instead of $\ds$ for point-tuples with different base points, the distance
\begin{equation}\label{eq:DefDpt}
	\dpt ([x,y],[u,v]) = \big\| \log_x(y) - \pt_x(\log_u(v))\big\|_x,
\end{equation}
where $\text{pt}_x(z)$ is the parallel transport of $z \in T\Mc$ to $T_x \Mc$, and a similar adaption of $\ds^\sym$ for bivariate signals. It was shown in \cite{holler18tgvm_mh} that also this version suitably generalizes TGV by transferring some of its main properties to the manifold setting.

Another existing extension of TGV to the manifold setting is the one presented in \cite{bergmann2018ictv_tgv_mh} which is given, in the univariate setting, as
\[ 
\widetilde{\TGVat}(u) = \inf _{(\xi_i)_i} \sum\nolimits_i \alpha_1 \| \log_{u_i} (u_{i+1}) - \xi_i \|_{u_i} + \alpha_0\| \xi_i  - P_{u_i} (\xi_{i-1}) \|_{u_i} \]
where $P_{u_i}$ approximates the parallel transport of $\xi_{i-1} $ to $u_i$ by first mapping it down to a point tupel $[u_{i-1},\exp_{u_{i-1}} (\xi_i) ]$, then using the pole ladder \cite{lorenzi14pole_ladder_mh} as an alternative to Schild's ladder to approximate the parallel transport to $u_i$ and finally lifting the transported tuple again to the tangent space via the logarithmic map. In the univariate case, this also generalizes TGV and preserves its main properties such as a well defined kernel. For the bivariate version, \cite{bergmann2018ictv_tgv_mh} uses the standard Jacobian instead of the symmetrized derivative and it remains open to what extend the kernel of TGV is appropriately generalized, also because there is no direct, natural generalization of the kernel of $\TGVat$ (i.e., affine functions) in the bivariate setting (see the next paragraph for details).\\

\noindent \textbf{Consistency.} Given that there are multiple possibilities of extending vector-space concepts to manifolds, the question arises to what extend the extensions of $\TV^2$, $\IC$ and $\TGV$ 
presented above are natural or ``the correct ones.'' As observed in \cite{holler18tgvm_mh}, the requirement of suitably transferring the kernel of the vector-space version, which consists of the set of affine functions, is a property that at least allows to reduce the number of possible generalizations. Motivated by this, we consider the zero-set of the manifold extensions of $\TV^2$, $\IC_\alpha$ and $\TGVat$. We start with the univariate situation, where a generalization of the notion of ``affine'' is rather natural. 
\begin{definition}[Univariate generalization of affine signals] Let $u = (u_i)_i$ be a signal in $\Mc^N$. We say that $u$ is generalized affine or \emph{geodesic} if there exists a geodesic $\gamma :[0,L] \rightarrow \Mc$ such that all points of $u$ are on $\gamma$ at equal distance and $\gamma$ is length-minimizing between two subsequent points.
\end{definition}
 
The following proposition relates geodesic functions to the kernel of higher-order regularizers on manifolds. Here, in order to avoid ambiguities arising from subtle difference in the functionals depending on if length-minimizing geodesics are used or not, we assume geodesics are unique noting that the general situation is mostly analogue.
\begin{proposition}[Consistency, univariate] \label{prop:consistency_univariate} Let $u = (u_i)_i$ in $\Mc$ be such that all points $u_i,u_j$ are connected by a unique geodesic.
\begin{enumerate}[(i)]
\item If $u$ is geodesic, $\TV^2(u)  = \IC_\alpha(u) = \STGVat(u) = 0$.%

\item  Conversely, if $\TV^2(u) = 0$ or $\STGVat(u) = 0$
, then $u$ is geodesic. 

\end{enumerate}
\begin{proof}
If $u$ is geodesic, it follows that $u_i  \in [u_{i-1},u_{i+1}]_{\frac{1}{2}}$, such that $\TV^2(u) = 0$. In case of $\IC_\alpha$ %
define $v_i = u_1$ (the first point of $u$) for all $i$ and $w_i = [u_1,u_i]_{2d(u_1,u_i)}$. Then it follows that $u_i \in [v_i,w_i]_{\frac{1}{2}}$ for all $i$. Further, $\TV((v_i)_i) = 0$ and, since $(w_i)_i$ is geodesic, also $\TV^2((w_i)_i) = 0$ such that $\IC_\alpha(u) = 0$. Regarding $\STGVat$, we see that  $\STGVat(u) = 0$ follows from choosing $(y_i)_i = (u_{i+1})_i$ and noting that $\ds([u_i,u_{i+1}],[u_{i-1},u_i]) = 0$ since $u_i \in [u_i,u_i]_ {\frac{1}{2}}$ and $u_{i+1} \in [u_{i-1},u_i]_2$. 
Now conversely, if $\TV^2(u) = 0$, it follows that $u_i \in [u_{i-1},u_{i+1}]$ for all $i$ such that $(u_i)_i$ is geodesic. If $\STGVat(u) = 0$ we obtain $(y_i)_i = (u_{i+1})_i$ and, consequently, $u_{i+1} \in [u_i,u_{i-1}]$ which again implies that $u$ is geodesic. 
\end{proof}
\end{proposition}

\begin{remark} One can observe that in Proposition \ref{prop:consistency_univariate}, the counterparts of ii) for $\IC_\alpha$ is missing. Indeed, an easy counterexample shows that this assertion is not true, even in case of unique geodesics: Consider $\Mc = \Sc^2 \cap ([0,\infty) \times \RR \times [0,\infty))$ and $\phi_1 = -\pi/4,\phi_2 = 0, \phi_3 = \pi/4$ and $\psi = \pi/4$ define 
\[ u_i = (\cos(\varphi_i)\sin(\psi),\sin(\varphi_i)\sin(\psi),\cos(\psi))\]
\[ w_i = (\cos(\varphi_i),\sin(\varphi_i),0)\]
and $v_i = (0,0,1)$ for all $i$. Then $u_i \in [v_i,w_i]_{\frac{1}{2}}$, $\TV(v) = 0$, $\TV^2(w) = 0$, hence $\IC_\alpha(u) = 0$ but $u$ is not geodesic.
\end{remark}

In the bivariate setting, a generalization of an affine function is less obvious: It seems natural that $u=(u_{i,j})_{i,j}$ being \emph{generalized affine} or, in the notion above, \emph{geodesic} should imply that for each $k$, both $(u_{i,k})_i$ and $(u_{k,j})_j$ are geodesics. However, to achieve a generalization of affine, an additional condition is necessary to avoid functions of the form $(x,y) \mapsto xy$. In \cite{holler18tgvm_mh} this condition was to require also the signal $(u_{i+k,j-k})_k$ to be geodesic for each $i,j$. While this has the disadvantage favoring one particular direction, it is less restrictive than to require, in addition, also $(u_{i+k,j+k})_k$ to be geodesic. As shown in the following proposition, bivariate functions that are geodesic in the sense of \cite{holler18tgvm_mh} coincide with the kernel of $\STGVat$. $\TV^2$ on the other hand, gives rise to a different notion of affine.

\begin{proposition}[Kernel, bivariate] \label{prop:kernel_bivariate}
Let $u = (u_{i,j})_{i,j}$ be such that all points of $u$ are connected by a unique geodesics.
\begin{enumerate}[(i)]
\item $\STGVat(u) = 0$ if and only if $(u_{k,j_0})_k$, $(u_{i_0,k})_k$ and $(u_{i_0+k,j_0-k})_k$ is geodesic for each $i_0,j_0$. 
\item $\TV^2(u) = 0$ if and only if $(u_{k,j_0})_k$, $(u_{i_0,k})_k$ is geodesic for each $i_0,j_0$ and $ [u_{i+1,j},u_{i,j-1}]_{\frac{1}{2}} \cap [u_{i,j},u_{i+1,j-1}]_{\frac{1}{2}} \neq \emptyset $ for all $i,j$.
\end{enumerate}
\begin{proof}
For $\STGVat$, a stronger version of this result is proven in \cite[Theorem 2.18]{holler18tgvm_mh}. 
For $\TV^2$, this follows analogously to the univariate case directly from the definition of $\dc$ and $\dcc$. 
\end{proof}
\end{proposition}

\noindent \textbf{Higher-order regularized denoising.} The next proposition shows that $\TV^2$ and $\STGV$ based denoising of manifold valued data is indeed well-posed.

\begin{proposition}\label{prop:tv2_tgv_well_posed} Both $\TV^2$ and $\STGVat$ are lower semi-continuous w.r.t. convergence in $(\Mc,\dist)$. Further, for $R \in \{\alpha\TV^2,\STGVat\}$, where $\alpha>0$ the the case of $\TV^2$, the problem
\[ \inf_{u=(u_{i,j})_{i,j}} R(u)  + \sum\nolimits_{i,j} d(u_{i,j},f_{i,j})^q \]
admits a solution.
\end{proposition}
\begin{proof}
The proof is quite standard: by the Hopf-Rinow theorem, it is clear that the claimed existence follows once lower semi-continuity of $R$ can be guaranteed. For $\STGVat$, this is the assertion of \cite[Theorem 3.1]{holler18tgvm_mh}. For $\TV^2$, it suffices to show lower semi-continuity of $\dc$ and $\dcc$. We provide a proof for $\dc$, the other case works analogously. Take $u^n = (u_-^n,u_\circ^n,u_+^n)_n$ converging to $(u_-,u_\circ,u_+)$ and take $(c_n)_n$ with $c_n \in  [u^n_-,u^n_+]$ such that $d(c_n,u_\circ^n) \leq \inf_{c \in [u^n_-,u^n_+]} d(c,u^n_\circ) + 1/n$. Then, from boundedness of $(u^n)_n$ and since $d(c^n,u_-^n) \leq d(u_+^n,u_-^n)$ we obtain boundedness of $(c^n)_n$, hence a (non-relabeled) subsequence converging to some $c \in \Mc$ exists. From uniform convergence of the geodesics $\gamma ^n:[0,1] \rightarrow \Mc$ connecting $u^n_-$ and $u^n_+$ such that $c^n =  \gamma^n(1/2)$ (see for instance \cite[Lemma 3.3]{holler18tgvm_mh}) we obtain that $c \in [u_-,u_+]_{\frac{1}{2}}$ such that
\( \dc(u_-,u_\circ,u_+) \leq d(c,u_\circ) \leq \liminf_n d (c_n,u_\circ^n) \leq \liminf _n   \dc(u^n_-,u_\circ^n,u_+^\circ).\)
\end{proof}

\begin{remark} We note that the arguments of Proposition \ref{prop:tv2_tgv_well_posed} do not apply to $\IC$ since we cannot expect $(v_{i,j})_{i,j}$ and $(w_{i,j})_{i,j}$ with $u_{i,j} \in [v_{i,j},w_{i,j}]_{\frac{1}{2}}$ to be bounded in general. Indeed, this is similar to vector-space infimal convolution, only that there, one can factor out the kernel of $\TV$ and still obtain existence. %
\end{remark}

\subsection{Algorithmic Realization}
\label{sec:higherOandTGValgo}
Here we discuss the algorithmic realization and illustrate the results of $\TV^2$ and $\STGV$ regularized denoising; for $\IC$ we refer to \cite{Steidl17_infcon_manifold_mh,bergmann2018ictv_tgv_mh}. We note that, in contrast to the $\TV$ functional, the $\TV^2$ and the $\STGVat$ functional are not convex on Hadamard manifolds (cf. \cite[Remark 4.6]{Bavcak2016tv2_manifold_mh}) such that we cannot expect to obtain numerical algorithms that provably converge to global optimal solutions as for $\TV$ in Hadamard spaces (cf. Theorem \ref{thm:ConvergenceAlgA}). Nevertheless, the cyclic proximal point algorithm and the parallel proximal point algorithm as described in Section \ref{sec:TValgo} are applicable in practice. In the following, we discuss the corresponding splittings and proximal mappings, where we focus on the univariate case since, similar to Section \ref{sec:TValgo}, the bivariate case for $p=1$ can be handled analogously; for details we refer to \cite{Bavcak2016tv2_manifold_mh,holler18tgvm_mh}. 

For $\TV^2$ denoising, we may use the splitting
\begin{align*}
\tfrac{1}{q} \sum\nolimits_{i} \dist(u_i,f_i)^q + \alpha\TV^2((u_i)_i) = F_0(u) + F_1(u) + F_2(u) + F_3(u)
\end{align*}
where
\( F_0(u) = \tfrac{1}{q}\sum\nolimits_{i} \dist(u_i,f_i)^q,\)
	and
\[
 \qquad F_j(u) = \sum\nolimits_i \dc(u_{3i - 1 + j},u_{3i + j},u_{3i+1+j}), \qquad j=1,\ldots,3.
\]
Due to the decoupling of the summands, the computation of the proximal maps of $(F_i)_{i=0}^3$ reduces to the computation of the proximal maps of $x \mapsto \dist(x,f)^q$ and of $(x_1,x_2,x_3) \mapsto \dc(x_1,x_2,x_3)$. The former is given explicitly as in \eqref{eq:tv_manifoldProxData}, while for the latter, following \cite{Bavcak2016tv2_manifold_mh}, we use a subgradient descent scheme (see for instance \cite{ferreira1998subgradient_mh}) to iteratively solve
\[ \min _{x_{k-1},x_k,x_{k+1}}  \frac{1}{2}\sum\nolimits_{l=k-1}^{k+1} \dist^2 (x_l,h_l) + \lambda \dc (x_{k-1},x_k,x_{k+1})
\]
for $(h_{k-1},h_k,h_{k+1}) \in \Mc^3$ the given point where the proximal map needs to be computed.

\begin{algorithm}[t]
  \begin{algorithmic}[1]

\State \textbf{SGD}($x_0,(\lambda_i)_i$)
\State $l=0$
\State \quad \textbf{repeat} until stopping criterion fulfilled
	\State \qquad $x_{k+1} \gets \exp_{x_k} \left( - \lambda_k \partial F(x_k) \right)  $
	\State \qquad $k \gets k+1$
\State \Return{$x$}
\end{algorithmic}
\caption{Subgradient descent for solving $\min_x F(x)$}\label{alg:subgradient_descend}
\end{algorithm}
For this purpose, we employ Algorithm \ref{alg:subgradient_descend}, which requires to compute the subgradient of $\dc$ as well as the derivative of $\dist$. The latter is, at a point $(x_{k-1},x_k,x_{k+1})$ given as \\
\(- \left( \begin{matrix}
\log_{x_{k-1}}(h_{k-1}), 
\log_{x_{k}}(h_{k}),  
\log_{x_{k+1}}(h_{k+1})  
\end{matrix} \right)^T.
\)
Regarding the computation of $\dc$, in order to avoid pathological (and practically irrelevant) constellations, we assume that the arguments $x_{k-1},x_{k+1}$ are not cut points, such that there is exactly one length minimizing geodesics connecting $x_{k-1}$ and $x_{k+1}$ and the corresponding midpoint is unique. 
With these assumptions, the derivative w.r.t. the first component and for a point $(x,y,z)$ with $y \neq [x,z]_{\frac{1}{2}}$ is given as
\[ \partial _{y} \dc(x,\cdot,z) (y) = {\log_{y}([x,z]_{\frac{1}{2}})}/{\|\log_{y}([x,z]_{\frac{1}{2})}\|_{y}}.\]
The derivative w.r.t. $x$ is given by
\[ \partial _{x} \dc(\cdot,y,z) (x) = \sum\nolimits_{l=1}^d \left \langle {\log_{c}(y) }/{\|\log_{c}(y) \|_{c}}, D_x c(\xi_l) \right \rangle \]
where we denote $c = [x,z]_{\frac{1}{2}}$. 
Here, $D_xc$ is the differential of the mapping $c: x \mapsto  [x,z]_{\frac{1}{2}}$ which is evaluated w.r.t.\ the elements of an orthonormal basis $(\xi_l)_{l=1}^d$ of the tangent space at $x.$
 The derivative w.r.t. $z$ is computed analoguosly due to symmetry and for the particular case that $y = [x,z]_{\frac{1}{2}}$ we refer to \cite[Remark 3.4]{Bavcak2016tv2_manifold_mh}. While the formulas above provide general forms of the required derivatives, a concrete realization can be done using explicit formulae for Jacobi fields in the particular manifold under consideration; we refer to \cite{Bavcak2016tv2_manifold_mh} for explicit versions. For the biviariate case, we also refer to \cite{Bavcak2016tv2_manifold_mh} for the derivative of $\dcc$, which can be computed with similar techniques as $\dc$.

\begin{figure}[t]
	\centering
	\def\figfolder{experiments/compare_with_TV_2D_S2_5/}
	\def\figurewidth{0.3\textwidth}
	\def\hs{\hspace{0.03\textwidth}}
	\begin{subfigure}[t]{\figurewidth}
		\includegraphics[width=1\textwidth]{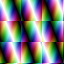}
		\caption{Original.}
	\end{subfigure}
	\hs 
	\begin{subfigure}[t]{\figurewidth}
		\includegraphics[width=1\textwidth]{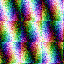}
		\caption{Noisy data.}
	\end{subfigure}
	\hs 
	\begin{subfigure}[t]{\figurewidth}
		\includegraphics[width=1\textwidth]{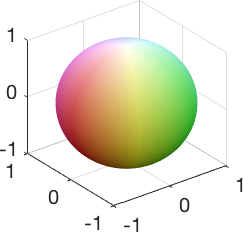}
		\caption{Color code for $S^2$.}
	\end{subfigure}
	\\[2ex]
	\begin{subfigure}[t]{\figurewidth}
		\includegraphics[width=1\textwidth]{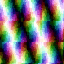}
		\caption{TV, $\deltaSNR = \protect\input{\figfolder deltaSNR_TV.txt} \dB$ }
	\end{subfigure} 
	\hs
	\begin{subfigure}[t]{\figurewidth}
		\includegraphics[width=1\textwidth]{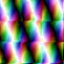}
		\caption{$\TV^2$, $\deltaSNR = \protect\input{\figfolder deltaSNR_TV2.txt} \dB$ }
	\end{subfigure} 
	\hs
	\begin{subfigure}[t]{\figurewidth}
		\includegraphics[width=1\textwidth]{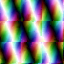}
		\caption{S-TGV , $\deltaSNR = \protect\input{\figfolder deltaSNR_TGV.txt} \dB$}
	\end{subfigure}
	\caption{Comparison of first and second-order total variation as well as $\STGV$
		on an $S^2$-valued image from~\cite{Bavcak2016tv2_manifold_mh}. Images taken from \cite{holler18tgvm_mh}.
	}
	\label{fig:S2_synth}
\end{figure}

For TGV, we again start with the univariate case and consider the $\STGVat$ variant. We consider the splitting
\begin{align*}
\tfrac{1}{q} \sum\nolimits_{i} \dist(u_i,f_i)^q + \TGV_\alpha^2((u_i)_i) = F_0(u) + F_1(u) + F_2(u) + F_3(u)
\end{align*}
where $F_0(u) = \frac{1}{q}\sum_{i} \dist(u_i,f_i)^q,$ $F_1(u) = \sum_{i} \dist(u_{i+1},y_i),$ and 
\begin{align*}
 F_2(u) & = \sum_{ \substack{ i:i \text{ even}}} \ds([u_{i},y_{i}],[u_{i-1,}y_{i-1}]),\qquad   F_3(u) = \sum_{ \substack{ i:i \text{ odd}}} \ds([u_{i},y_{i}],[u_{i-1}y_{i-1}]). 
 \end{align*}
Here, as an advantage of this particular version of TGV that uses only points on the manifold, we see that the proximal mappings of $F_1$ are explicit as in \eqref{eq:tv_manifoldProxReg} and, since again the proximal mapping of $F_0$ is given for $q \in \{1,2\}$ explicitly by \eqref{eq:tv_manifoldProxData}, we are left to compute the proximal mappings for $\ds$. To this aim, we again apply the subgradient descent method as in Algorithm \ref{alg:subgradient_descend}, where the required derivatives of $\ds$ are provided in \cite{holler18tgvm_mh}. In particular, again assuming uniqueness of geodesics to avoid pathological situations, we have that for points $[u_i,y_i],[u_{i-1},y_{i-1}]$ with $y_{i}\neq [u_{i-1},[u_{i},y_{i-1}]_{\tfrac{1}{2}} ]_2$ that
\begin{equation}
\nabla_{y_{i}} \ds = 
- \log_{y_{i}}S(u_{i-1},y_{i-1},u_{i})  /
{\big \|\log_{y_{i}}S(u_{i-1},y_{i-1},u_{i})   \big \|}
\end{equation}
where 
\begin{equation}
 S(u_{i-1},y_{i-1},u_{i}) = [u_{i-1},[u_{i},y_{i-1}]_{\tfrac{1}{2}} ]_2
\end{equation}
denotes the result of applying the Schild's construction to the respective arguments.
Derivatives w.r.t. the arguments $u_{i-1}$ and $y_{i-1}$ are given in abstract form as
\begin{align}\label{eq:InitDiscussion1}
\nabla_{u_{i-1}} \ds &=  
-T_1
\left(
{\log_{S(u_{i-1},y_{i-1},u_{i})} y_{i}}    
/
{\big \| \log_{S(u_{i-1},y_{i-1},u_{i})} y_{i} \big \|}
\right), \\
\nabla_{y_{i-1}} \ds &=  
-T_2
\left(
{\log_{S(u_{i-1},y_{i-1},u_{i})} y_{i}}    
/
{\big \| \log_{S(u_{i-1},y_{i-1},u_{i})} y_{i} \big \|}
\right),\label{eq:InitDiscussion2}
\end{align}
where $T_1$ is the adjoint of the differential of the mapping 
$u_{i-1} \mapsto 
[u_{i-1},[u_{i},y_{i-1}]_{{1}/{2}} ]_2$,
and 
$T_2$ is the adjoint of the differential of the mapping 
$y_{i-1} \mapsto 
[u_{i-1},[u_{i},y_{i-1}]_{{1}/{2}} ]_2$.
The differential w.r.t. $u_{i}$ is obtained by symmetry. Again, the concrete form of these mappings depends on the underlying manifold and we refer to \cite{holler18tgvm_mh} for details. Regarding points with $y_{i}\neq [u_{i-1},[u_{i},y_{i-1}]_{{1}/{2}} ]_2$ we note that for instance the four-tuple consisting of the four zero-tangent vectors sitting in  $u_i,u_{i-1},[u_{i-1},[u_i,y_{i-1}]_{{1}/{2}}]_2,$ $y_{i-1}$ belongs to the subgradient of $\ds.$
The algorithm for bivariate TGV-denoising can be obtained analogously, where we refer to \cite{holler18tgvm_mh} for the computation of the derivative of $\ds^\text{sym}$. 
The algorithm for TGV-denoising based on the parallel variant \eqref{eq:DefDpt} employs the proximal mappings of $F_0$ and $F_1$ as well. Implementation of the proximal mappings of $F_2$ and $F_3$ based on subgradient descent can be found in \cite{holler18tgvm_mh} and \cite {bredies2018Observation}.

\noindent{\bf Numerical examples.}
We illustrate the algorithm with numerical examples. At first, we provide a comparison between TV, $\TV^2$ and $\STGV$ regularization for synthetic $\mathbb S^2$-valued image data, taken from \cite{holler18tgvm_mh}. The results can be found in Figure \ref{fig:S2_synth}, where for each approach we optimized over the involved parameters to achieve the best $\deltaSNR$ result, with $\deltaSNR$ being defined for ground truth, noisy and denoised signals $h$, $f$ and $u$, respectively, as
\(
	\deltaSNR = 10 \log_{10} \left(  \tfrac{\sum_{i} \dist(h_{i}, f_{i})^2   }{\sum_{i} \dist(h_{i}, u_{i})^2}\right) \dB.
\)
We observe that the TV regularization produces significant piecewise constancy artifacts (staircasing) on the piecewise smooth example. The result of $\TV^2$ shows no visible staircasing, but smoothes the discontinuities to some extend. $\STGV$ is able to better reconstruct sharp interfaces while showing no visible staircasing. 

As second example, Figure \ref{fig:InSAR} considers the denoising of interferometric synthetic aperture radar (InSAR) images. Again, it can be oberserved that $\STGV$ has a significant denoising effect while still preserving sharp interfaces. 

\begin{figure}[t]
\centering
\def\figfolder{experiments/InSAR_3/}

\centering
\def\figurewidth{0.68\textwidth}

\tikzstyle{myspy}=[spy using outlines={black,lens={scale=5},width=0.3\textwidth, height=0.3\textwidth, connect spies, every spy on node/.append style={thick}}]
\centering
\def\subfigwidth{0.45\textwidth}
\def\nodeSpy{(0.5, 0.9)}
\def\nodeSpyB{(0.6, -0.8)}
\def\nodeWindow{(4.4,1.12)}
\def\nodeWindowB{(4.4,-1.12)}

\begin{subfigure}[t]{\subfigwidth}
	\begin{tikzpicture}[myspy]
	\node {\includegraphics[interpolate=false,width=\figurewidth]{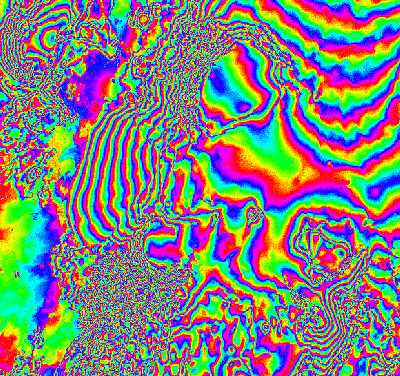}};
	\spy on \nodeSpy in node [left] at \nodeWindow;
	\spy on \nodeSpyB in node [left] at \nodeWindowB;
\end{tikzpicture}
\end{subfigure}
\hspace*{.5cm}
\begin{subfigure}[t]{\subfigwidth}
	\begin{tikzpicture}[myspy]
	\node {\includegraphics[interpolate=false,width=\figurewidth]{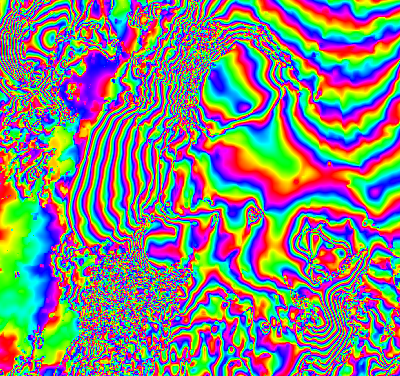}};
	\spy  on \nodeSpy in node [left] at \nodeWindow;
	\spy on \nodeSpyB in node [left] at \nodeWindowB;
	\end{tikzpicture}
\end{subfigure}
	\caption{
	\emph{Left:} InSAR image from \cite{thiel1997ers}. 
	\emph{Right:} Result of S-TGV. Image taken from \cite{holler18tgvm_mh}.
	}
\label{fig:InSAR}
\end{figure}

\section{Mumford-Shah Regularization for Manifold Valued Data}
\label{sec:MumSha}

The Mumford and Shah model \cite{mumford1985boundary,mumford1989optimal}, also called Blake-Zisserman model \cite{blake1987visual}, is a powerful variational model 
for image regularization. The regularization term measures the length of the jump set and, within segments, it measures quadratic variation of the function.  
The resulting regularization is a smooth approximation to the image/signal 
which, at the same time, allows for sharp edges at the discontinuity set.
Compared with the TV regularizer, it does not penalize the jump height and, due to the quadratic variation on the complement of the edge set, no staircasing effects appear.
(Please note that no higher order derivatives/differences are involved here.)
The piecewise constant variant of the Mumford-Shah model (often called Potts model)
considers piecewise constant functions (which then have no variation on the segments) and 
penalizes the length of the jump sets. Typical applications of these functionals
are smoothing and the use within a segmentation pipeline.
For further information considering these problems for scalar data from various perspectives (calculus of variation, stochastics, inverse problems) we refer to  
\cite{potts1952some, blake1987visual, geman1984stochastic, ambrosio1990approximation,chambolle1995image, wittich2008complexity,boysen2009consistencies,fornasier2010iterative,fornasier2013existence,jiang2014regularizing} 
and the references therein.  
These references also deal with fundamental questions such as the existence of minimizers.
Mumford-Shah and Potts problems are computationally challenging
since the functionals are non-smooth and non-convex. 
Even for scalar data, both problems are known to be NP-hard in dimensions higher than one \cite{veksler1999efficient, boykov2001fast, alexeev2010complexity}.
This makes finding a (global) minimizer infeasible.
Because of its importance in image processing however,
many approximative strategies have been proposed  
for scalar- and vector valued data.
Among these strategies are graduated non-convexity \cite{blake1987visual}, approximation by elliptic functionals \cite{ambrosio1990approximation}, 
graph cuts \cite{boykov2001fast}, active contours \cite{tsai2001curve}, 
convex relaxations \cite{pock2009algorithm}, iterative thresholding approaches \cite{fornasier2010iterative}, 
and alternating direction methods of multipliers \cite{hohm2015algorithmic}.

In the context of DTI, the authors of \cite{wang2005dti, wang2004affine} consider 
a Chan-Vese model for positive matrix-valued data, i.e., for manifold-valued data in $\mathrm{Pos}_3,$ 
as well as a piecewise smooth variant.
(We recall that Chan-Vese models are variants of Potts models for the case of two labels.)
Their method is based on a level-set active-contour approach.
In order to reduce the computational load in their algorithms
(which is due to the computation of Riemannian means for a very large number of points)
the authors resort to non-Riemannian distance measures in \cite{wang2005dti, wang2004affine}. 
Recently, a fast recursive strategy for computing the Riemannian mean has been proposed   
and applied to the piecewise constant Chan-Vese model (with two labels) in \cite{cheng2012efficient}.

We mention that for $\mathbb S^1$-valued data,  a noniterative exact solver for the univariate Potts problem has been proposed in \cite{storath2017jump}.

In this section,
we consider Mumford-Shah and Potts problems for (general) manifold-valued data
and derive algorithms for these problems. 
As in the linear case, typical applications of these functionals
are smoothing and also segmentation; more precisely, they can serve as an initial step of a segmentation pipeline. In simple cases, the induced edge set may yield a reasonable segmentation directly.

\subsection{Models}
\label{sec:MumShamodel}

We start out with Mumford-Shah and Potts problems for univariate manifold-valued data $(f_i)_{i=1}^N \in \Mc^N$, with $\Mc$ again being a complete, finite dimensional Riemannian manifold. 
The discrete Mumford-Shah functional reads 
\begin{equation} \label{eq:1dMS_mani_jumpFormulation}
\begin{split}
B_{\alpha,\gamma}(x) = \frac1q \sum_{i = 1}^N \dist(x_i,f_i)^q &+ \frac{\alpha}{p}  \sum_{i \notin \mathcal{J}(x)} \dist(x_i,x_{i+1})^p  + \gamma \left|\mathcal{J}(x)\right|,
\end{split}
\end{equation}
where $\dist$ is the distance with respect to the Riemannian metric in the manifold 
$\mathcal M$, $\mathcal{J}$ is the jump set of $x$ and $p,q \in [1,\infty)$. 
The jump set is given by $\mathcal{J}(x) = \{i: 1 \leq i < n $ and $ \dist(x_i,x_{i+1}) > s\},$
and $|\mathcal{J}(x)|$ denotes the number of jumps.
The jump height $s$ and the parameter $\gamma$ are related via $\gamma = \alpha s^p/p$.
We rewrite \eqref{eq:1dMS_mani_jumpFormulation}
using a truncated power function 
to obtain the Blake-Zisserman type form 
\begin{equation} \label{eq:1dMS_manifold_truncatedFormulation}
B_{\alpha,s}(x) = \frac1q \sum_{i = 1}^N \dist(x_i,f_i)^q + \frac{\alpha}{p}  \sum_{i=1}^{N-1} 
\min(s^p,\dist(x_i,x_{i+1})^p),
\end{equation}
where $s$ is the argument of the power function $t \mapsto t^p$, which is truncated at $\dist(x_i,x_{i+1})^p$.
The discrete univariate Potts functional for manifold-valued data reads
\begin{equation} \label{eq:1dPotts_mani}
P_{\gamma}(x) = \frac1q \sum_{i = 1}^n \dist(x_i,f_i)^q + \gamma |\mathcal{J}(x)|,
\end{equation}
where an index $i$ belongs to the jump set of $x$
if $x_i \neq x_{i+1}.$ 

In the multivariate situation, 
the discretization of the Mum\-ford-Shah and Potts problem
is not as straightforward as in the univariate case.
A simple finite difference discretization
with respect to the coordinate  directions
is known to produce undesired block 
artifacts in the reconstruction 
\cite{chambolle1999finite}.
The results improve significantly 
when including further finite differences 
such as the diagonal directions \cite{chambolle1999finite,storath2014fast, storath2014joint}.
To define bivariate Mumford-Shah and Potts functionals, we introduce the no\-ta\-tion~$\dist^q(x,y)$ $= \sum_{i,j}\dist^q(x_{ij}, y_{ij})$ for the $q$-distance of two manifold-valued images $x,y \in \mathcal M^{N\times M}.$
For the regularizing term, we employ the penalty function
\[
\Psi_{a}(x) = \sum\nolimits_{i,j} \psi(x_{(i,j) + a}, x_{ij}),
\]
where  $a \in \mathbb Z^2 \setminus \{0\}$ denotes a direction,
and the potentials $\psi$ are given by
\begin{align}\label{eq:psiMumfiPotts}
\psi(w,z) =  \frac{1}{p}\min(s^p, \dist(w, z)^p),\quad \text{ and} \quad
\psi(w,z) = 
\begin{cases}
1, &\text{if }w \neq z,\\
0, &\text{if }w = z,\\
\end{cases}
\end{align}
for $w, z \in \mathcal M$,
in the  Mumford-Shah case and in the Potts case, respectively.
We define the discrete multivariate Mumford-Shah and Potts problems by
\begin{equation}\label{eq:msDiscrete}
\min_{x \in \mathcal M^{N\times M}}  \frac{1}{q} \dist^q(x, f) + \alpha\sum_{s=1}^R \omega_s \Psi_{a_s}(x),
\end{equation}
where the finite difference vectors $a_s \in \mathbb Z^2\setminus\{0\}$ belong to a neighborhood system $\Nc$ and $\omega_1,..., \omega_R$ are
non-negative weights.
As observed in \cite{storath2014fast}, a reasonable neighborhood system is
\[
\Nc = \{  (1,0); (0,1); (1,1); (1,-1) \}
\]
with the weights $\omega_1 = \omega_2 = \sqrt{2} -1$ and $\omega_3 = \omega_4 = 1- \frac{\sqrt{2}}{2}$
as in \cite{storath2014fast}. It provides a sufficiently isotropic discretization while keeping the computational load at a reasonable level.
For further neighborhood systems and weights we refer to \cite{chambolle1999finite,storath2014fast}.

For both, the univariate and multivariate discrete Mumford-Shah and Potts functionals, the following result regarding the existence of minimizers holds.
\begin{theorem}\label{thm:ExMinimizers2D}
	Let $\mathcal M$ be a complete Riemannian manifold.
	Then the univariate and multivariate discrete Mumford-Shah and Potts problems 
	\eqref{eq:1dMS_mani_jumpFormulation}, \eqref{eq:1dPotts_mani}, 
	and \eqref{eq:msDiscrete}
	have a minimizer.
\end{theorem}
A proof may be found in \cite{weinmann2015mumford}.
We note that most of the data spaces in applications 
are complete Riemannian manifolds.

\subsection{Algorithmic Realization}
\label{sec:MumShaalgo}

We start with the univariate Mumford-Shah
and Potts problems   \eqref{eq:1dMS_mani_jumpFormulation} and \eqref{eq:1dPotts_mani}.
These are not only important on their own, variants also appear as subproblems in the algorithms for the multivariate problems discussed below.

{\em Dynamic programming scheme.}
To find a minimizer of the Mumford-Shah problem  \eqref{eq:1dMS_mani_jumpFormulation}
and the Potts problem  \eqref{eq:1dPotts_mani}, we employ a general dynamic programming scheme which was employed for related scalar and vectorial problems in various contexts \cite{mumford1989optimal,chambolle1995image,winkler2002smoothers,friedrich2008complexity,weinmann2014l1potts,
	storath2014jump}. 
We briefly explain the idea where we use the Mumford-Shah problem as example.
We assume that we have already computed minimizers $x^l$ of 
the functional $B_{\alpha,\gamma}$ associated with the partial data 
$f_{1:l} = (f_1,..., f_l)$ for each $l = 1,..., r-1$ and some $r \leq N.$
(Here, we use the notation
$
f_{l:r} := (f_l,..., f_r).)
$
We explain how to compute a minimizer $x^r$ associated to data $f_{1:r}.$
For each $x^{l-1}$ of length $l-1,$ 
we define a candidate $x^{l,r}=(x^{l-1},h^{l, r}) \in \mathcal M^r$
which is the concatenation of $x^{l-1}$	
with a vector $h^{l, r}$ of length $r - l +1;$ 
We choose $h^{l, r}$ as a minimizer of the problem
\begin{equation}\label{eq:epsilonGeneral}
\epsilon_{l,r} = \min_{h \in M^{r- l + 1}} \sum\nolimits_{i=l}^{r-1} \frac\alpha p \dist^p(h_i, h_{i+1}) + \frac1q \sum\nolimits_{i=l}^{r} \dist^q(h_i, f_i),
\end{equation}
where $\epsilon_{l,r}$ is the best approximation error on the (discrete) interval $(l,...,r).$
Then we calculate  
\begin{align}\label{eq:potts_value_candidate}
\min_{l=1,...,r} \left\{ B_{\alpha,\gamma}(x^{l-1}) +  \gamma + \epsilon_{l,r}   \right\},
\end{align}
which coincides with  
the minimal functional value  of $B_{\alpha,\gamma}$ for data $f_{1:r}.$
We obtain the corresponding minimizer $x^{r} = x^{l^*, r},$ 
where $l^*$ is a minimizing argument in \eqref{eq:potts_value_candidate}.
We successively compute $x^r$ for each $r = 1,..., N$ until we end up with full data $f.$
For the selection process, only the $l^*$ and the $\epsilon_{l,r}$ have to be computed;
the optimal vectors $x^{r}$ are then computed in a postprocessing step 
from these data; see, e.g., \cite{friedrich2008complexity} for further details.
This skeleton (without computing the $\epsilon_{l,r}$)
has quadratic complexity with respect to time and linear complexity with respect to space.	
In the concrete situation, it is thus important to find fast ways for computing the approximation errors $\epsilon_{l,r}.$
We will discuss this in the next paragraph for our particular situation. 
In practice, the computation is accelerated significantly using pruning strategies \cite{killick2012optimal,storath2014fast}.

\noindent \textbf{Algorithms for the univariate Mumford-Shah and Potts problem.}
To make the dynamic program work for 
the Mumford-Shah problem with manifold-valued data, we have to compute the approximation errors $\epsilon_{l,r}$
in \eqref{eq:epsilonGeneral}. These are $L^q$-$V^p$ type problems: the data term is a manifold $\ell^q$ distance and the second term is a $p$th variation; in particular, for $p=1$ we obtain TV minimization problems. These $L^q$-$V^p$ problems can be solves using the proximal point schemes discussed in Section~\ref{sec:TVmodel}; for  details, we refer to \cite{weinmann2014total} where in particular the corresponding proximal mappings are calculated in terms of gedesic averages for the important case of quadratic variation $p=2$.

\begin{figure}[t]
		\def\figfolder{experiments/Mumfi1D/}			
		
		\includegraphics[height=1.03\textwidth, angle=-90]{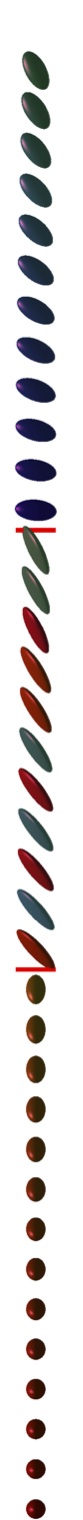}\\
		
		\includegraphics[height=1.03\textwidth, angle=-90]{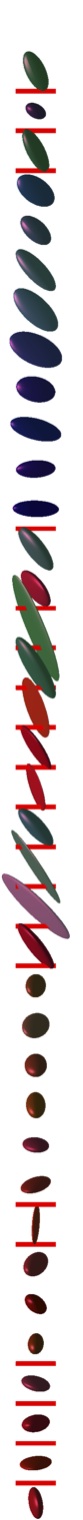}\\

		\includegraphics[height=1.03\textwidth, angle=-90]{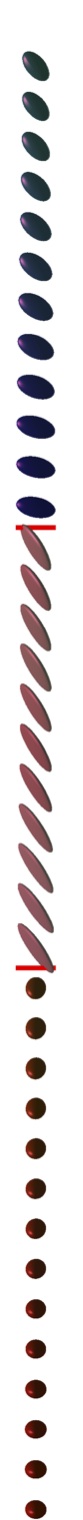}\\

	\caption{Univariate Mumford-Shah regularization ($p,q = 1$) using dynamic programming. (The red steaks indicate jumps.)
		{\em Top.} Original signal.
		{\em Middle.} Data with Rician noise added. %
		{\em Bottom.} Regularized signal.
		Mumford-Shah regularization 
		removes the noise while preserving the jump.}		
	\label{fig:ms_1D}  
\end{figure}
To make the dynamic program work for 
the Potts problem with manifold-valued data, we have to compute the approximation errors $\epsilon_{l,r}$ for the problem 
$
\epsilon_{l,r} = \min_{h \in \mathcal M^{r- l + 1}}  \frac1q \sum\nolimits_{i=l}^{r} \dist^q(h_i, f_i),
$
under the assumption that $h$ is a constant vector. Hence we have to compute 
\begin{equation} \label{eq:epsilonConc}
\epsilon_{l,r} = \min_{h \in \mathcal M} \frac1q\sum\nolimits_{i=l}^{r} \dist^q(h, f_i).
\end{equation}
We observe that, by definition, a minimizer of \eqref{eq:epsilonConc} is given by an intrinsic mean for $q=2,$ and by an intrinsic median for $q=1,$ respectively.

As already discussed, a mean is general not uniquely defined since the minimization problem has no unique solution in general. Further, there is no closed form expression in general. 
One means to compute the intrinsic mean is the gradient descent 
(already mentioned in \cite{karcher1977riemannian})
via the iteration
\begin{equation} \label{eq:GradientDescentIntMean}
h^{k+1} =  \exp_{h^{k}} \sum_{i=l}^r \tfrac{1}{r-l+1} \log_{h^{k}}f_i,
\end{equation}
where again $\log$ denotes the inverse exponential map.
Further information on convergence and other related topics can for instance be found in the papers \cite{fletcher2007riemannian,afsari2013convergence} and the references given there.
Newton's method was also applied to this problem in the literature; see, e.g., \cite{ferreira2013newton}. 
It is reported in the literature and also confirmed by the authors' experience that the gradient descent converges rather fast; 
in most cases, $5$-$10$ iterations are enough.
For general $p\neq 1,$ the gradient descent approach works as well.
The case $p=1$ amounts to considering the intrinsic median together with the intrinsic absolute deviation.
In this case, we may apply a subgradient descent
which in the differentiable part amounts to rescaling the tangent vector given on the right-hand side of \eqref{eq:GradientDescentIntMean} 
to length $1$ and considering variable step sizes which are square-integrable but not integrable; 
see, e.g., \cite{arnaudon2013approximating}.

A speedup using the structure of the dynamic program is obtained by 
initializing with previous output. More precisely, when starting the iteration of the mean for data 
$f_{l+1:r},$ we can use the already computed mean for the data $f_{l:r}$
as an initial guess. We notice that this guess typically becomes even better 
the more data items we have to compute the mean for, i.e., the bigger $r-l$ is. This is important since this case is the 
computational more expensive part and a good initial guess reduces the number of iterations needed. 

We have the following theoretical guarantees.  
\begin{theorem}\label{thm:AlgProducesMinimizers}
	In a Cartan-Hadamard manifold, the dynamic programming scheme 
	produces a global minimizer for the univariate Mumford-Shah problem \eqref{eq:1dMS_mani_jumpFormulation}
	and the discrete Potts problem \eqref{eq:1dPotts_mani}, accordingly.   
\end{theorem}
A proof can be found in \cite{weinmann2015mumford}. In this reference, also guarantees are obtained for Potts problems for general complete Riemannian manifold under additional assumptions; cf. \cite[Theorem 3]{weinmann2015mumford}. 
In Figure~\ref{fig:ms_1D}, the algorithm for the univariate case is illustrated  
for Mumford-Shah regularization for the Cartan-Hadamard manifold of positive matrices.

\noindent \textbf{Multivariate Mumford-Shah and Potts problems.}
We now consider Mumford-Shah and Potts regularization 
for manifold-valued images.
Even for scalar data, these problems are NP hard 
in dimensions higher than one even  \cite{veksler1999efficient, alexeev2010complexity}.
Hence, finding global minimizers is not tractable anymore in the multivariate case in general.
The goal is to derive approximative strategies that perform well in practice.
We present a splitting approach: we rewrite \eqref{eq:msDiscrete}
as the constrained problem
\begin{equation}\label{eq:constrained}
\min_{x_1,..., x_R}~\sum_{s=1}^R  \frac{1}{qR} \dist^q(x_s, f) + \alpha\omega_s \Psi_{a_s}(x_s)  \qquad
\text{s. t. }  x_s = x_{s+1},  \, s \in \{1,\ldots,R\}, 
\end{equation}
with the convention $x_{R+1} = x_1.$
We use a penalty method (see e.g.~\cite{bertsekas1976multiplier})
to obtain the unconstrained  problem
\begin{equation*}
\min_{x_1,..., x_R}~\sum\nolimits_{s=1}^R \omega_s qR \alpha \Psi_{a_s}(x_s) +  \dist^q(x_s, f)
+ \mu_k \dist^q(x_s, x_{s+1}).
\end{equation*}
We use an increasing coupling sequence $(\mu_k)_k$ which fulfills the summability condition
$\sum_k \mu_{k}^{-1/q} < \infty.$
This specific splitting allows us to minimize the functional block wise,
that is, with respect variables $x_1,..., x_R$ separately. 
Performing blockwise minimization yields the algorithm
\begin{align}\label{eq:OurSplitting}
\begin{cases}
x_1^{k+1} \in \argmin_{x}qR\omega_1 \alpha \Psi_{a_1}(x) + \dist^q(x, f)  
+ \mu_k \dist^q(x, x_{R}^k), 
 \\
x_2^{k+1} \in \argmin_{x}qR\omega_2 \alpha \Psi_{a_2}(x) +  \dist^q(x, f) 
+ \mu_k \dist^q(x, x_{1}^{k+1}), \\
\quad\vdots \\
x_R^{k+1} \in		\argmin_{x}qR\omega_R  \alpha\Psi_{a_R}(x) +  \dist^q(x, f) 
+ \mu_k \dist^q(x, x_{R-1}^{k+1}).
\end{cases}
\end{align}
We notice that each line of \eqref{eq:OurSplitting}
decomposes into univariate subproblems of Mumford-Shah and Potts 
type, respectively. The subproblems  are almost identical 
with the univariate problems above.
Therefore, we can use the algorithms developed above with a few minor modification.
Details may be found in \cite{weinmann2015mumford}.

There is the following result ensuring that the algorithm terminates.
\begin{theorem} \label{thm:Convergence2D}
	For Cartan-Hadamard manifold-valued images 
	the algorithm \eqref{eq:OurSplitting} for both the Mumford-Shah and the Potts problem converge.
\end{theorem}
A proof can be found in \cite{weinmann2015mumford}.

A result of the algorithm is illustrated in Figure~\ref{fig:ms_2D} 
for Mumford-Shah regularization in the Cartan-Hadamard manifold of positive matrices.
The data set was taken from the Camino project \cite{cook2006camino}.

\begin{figure}[t]
	\def\figfolder{experiments/Mumfi2D/}			
	
	\def\subfigwidth{0.48\textwidth}
	\def\figurewidth{0.65\textwidth}
	\def\figureheight{0.7\textwidth}
	\def\hs{\hspace{0.05\textwidth}}
		\includegraphics[width=.49\textwidth, trim=0 0 0 45, clip]
		{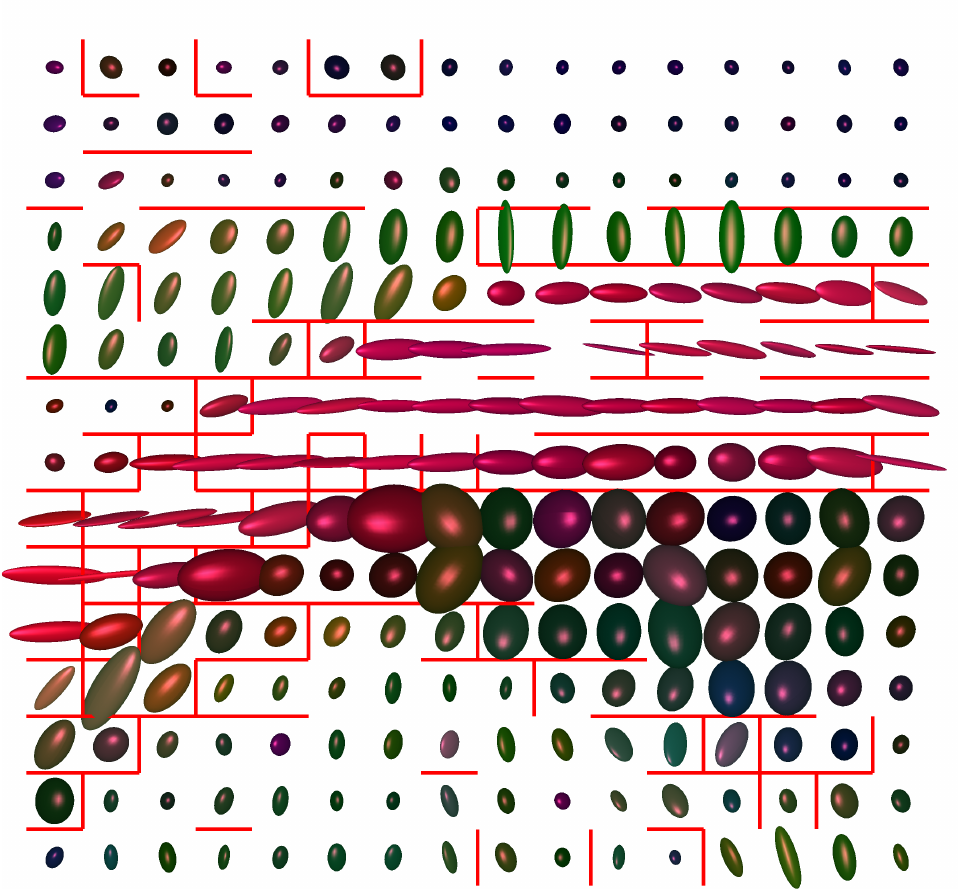} 		
			\raisebox{0.2ex}{
				\includegraphics[width=0.49\textwidth, trim=0 0 0 45, clip]
				{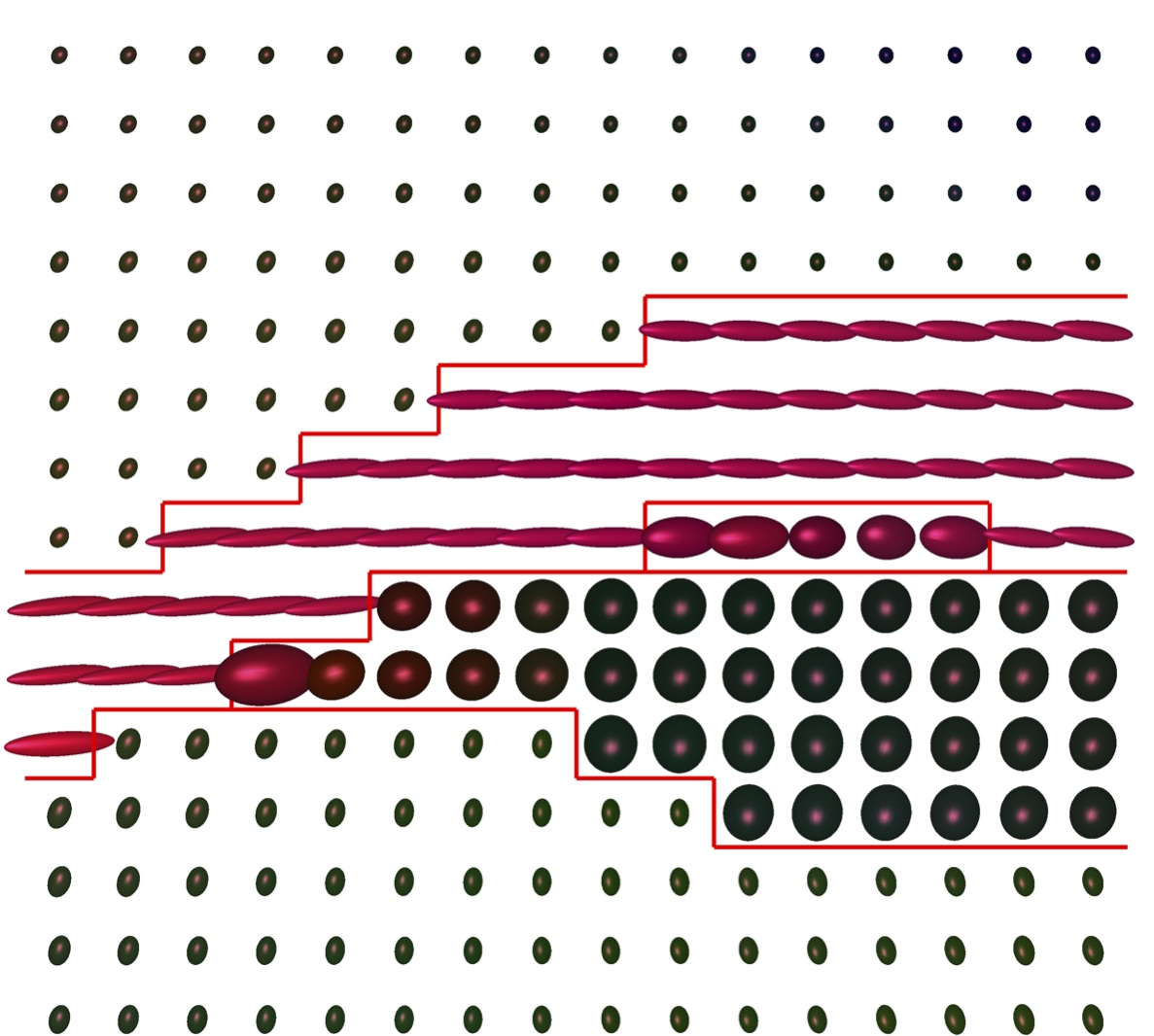}
				}
		
	\caption{{\em Left.} Part of a corpus callosum of a human brain \cite{cook2006camino}. 
		{\em Left:} Mumford-Shah regularization with $p,q=1.$ %
		The noise is significantly reduced and the edges are preserved. 
		Here, the edge set (depicted as red lines) of the regularization yields a segmentation.
	}
	\label{fig:ms_2D}
\end{figure}

\section{Dealing with Indirect Measurements: Variational Regularization of Inverse Problems for Manifold Valued Data}
\label{sec:InvProb}

In this section, we consider the situation when the data is not measured directly. More precisely, we consider the manifold valued analogue of the discrete inverse problem of reconstructing the signal $u$ in the equation 
$ \Ac u \approx f,$ with given noisy data $f$.
Here, $\Ac \in \RR^{K \times N} $ is a matrix with unit row sums (and potentially negative items), and $u$ is the objective to reconstruct. 
In the linear case, the corresponding variational model, given discrete data $f=(f_i)_{i=1}^K$,
reads
\begin{align}\label{eq:TichPhEuclidean}
\argmin_{u \in \mathbb{R}^N}  \frac{1}{q}\sum\nolimits_{i=1}^K \left|\sum\nolimits_{j=1}^N \Ac _{i,j}u_j-f_i \    \right|^q   +    R_\alpha(u).
\end{align} 
Here, the symbol $R_\alpha$ denotes a regularizing term incorporating prior assumption on the signal.
The process of finding $u$ given data $f$ via minimizing $\eqref{eq:TichPhEuclidean}$ is called Tikhonov-Phillips regularization.
For a general account on inverse problems and applications in imaging we refer to the books \cite{engl1996regularization,bertero1998introduction}.

In this section we consider models for variational (Tikhonov-Phillips) regularization
for indirect measurement terms in the manifold setup, 
we present algorithms for the proposed models
and we show the potential of the proposed schemes. The material is mostly taken from \cite{storath2018variational}.

\subsection{Models}
\label{sec:InvProbmodel}

We introduce models for variational (Tikhonov-Phillips) regularization
of indirectly measured data in the manifold setup. The approach is as follows: Given a matrix $\Ac = (\Ac_{i,j})_{i,j} \in \RR^{K \times N}$ with unit row sum, we replace the euclidean distance  in $|\sum\nolimits_j \Ac _{i,j}u_j-f_i \    |$ by the Riemannian distance $\dist(\cdot,f)$ in the complete, finite dimensional Riemannian manifold $\Mc$ and 
the weighted mean $\sum_j \Ac _{i,j}u_j$ by the weighted Riemannian center of mass \cite{karcher1977riemannian, kendall1990probability}
denoted by $\mean(\Ac_{i,\cdot},u)$
which is given by 
\begin{align}\label{eq:IntrMeanIntro}
\mean(\mathcal A_{i,\cdot},u) = \argmin_{v \in \mathcal{M}} \ \sum\nolimits_j  \mathcal A_{i,j} \ \dist(v,u_j)^2.
\end{align}
We consider the manifold analogue of the variational problem \eqref{eq:TichPhEuclidean} given by
\begin{align}\label{eq:ManiInvPro}
\argmin_{u \in \mathcal{M}^N} \frac1q\sum\nolimits_{i=1}^K \dist\left(\mean(\mathcal A_{i,\cdot},u),f_i\right)^q + \  R_\alpha(u).
\end{align}
Here $R_\alpha(u)$ is a regularizing term, for instance 
\begin{align}
R_\alpha(u) = \alpha \TV(u), \qquad \text{or} \qquad  R_\alpha(u) = \TGVat(u),
\end{align}
where $\TV(u)$ denotes the total variation as discussed in Section \ref{sec:TV},
and $\TGVat(u)$ denotes the total generalized variation 
for the discrete manifold valued target $u$ as discussed in Section \ref{sec:higherOandTGV},
for instance, the Schild variant and the parallel transport variant of TGV.

We note that also other regularizers $R$ such as the 
Mumford-Shah and Potts regularizers of Section~\ref{sec:MumSha} and the 
wavelet sparse regularizers of Section~\ref{sec:WavSparse} may be employed.

We point out that our setup includes the manifold analogue of convolution operators (a matrix with constant entries on the diagonals), e.g., modeling blur. Further, we notice that the discussion includes the multivariate setup (by serializing).

There are the following well-posedness results for the variational problems, i.e., 
results on the existence of minimizers.
For a general regularizer $R_\alpha$, under a coercivity type condition in the manifold setup
the existence of a minimizer is guaranteed as the following theorem shows.

\begin{theorem}\label{thm:ExistenceCondR}
	We consider a sequence of signals $(u^{k})_k$ in $\Mc^N$ and use the notation $\diam(u^{k})$ to denote the diameter of a single element $u_k$ (viewed as $N$ points in $\Mc$) of the sequence $\{ u_k \,| \, k \in \mathbb{N} \}$.
	If $R_\alpha$ is a regularizing term such that $R_\alpha(u^{k}) \to \infty,$ as $\diam(u^{k}) \to \infty,$ 
	and $R_\alpha$ is lower semicontinuous, then the variational problem \eqref{eq:ManiInvPro} with indirect measurement term has a minimizer.
\end{theorem}

This theorem is formulated as Theorem 1 in \cite{storath2018variational} and proved there. In particular, it applies 
to the $\TV$ regularizers and the their analogues considering $q${\em th} variation instead of total variation as well as mixed first-second order regularizers
of the form $\alpha_1 \TV + \alpha_0 \TV^2$ with $\alpha_1>0$, $\alpha_0 \geq 0$. 
\begin{theorem}\label{thm:ExistenceMinimizers}
	The inverse problem \eqref{eq:ManiInvPro} for manifold-valued data with $\TV$ regularizer 
	has a minimizer.
	The same statement  applies to mixed first and second order regularizers 
	of the form $\alpha_1 \TV + \alpha_0 \TV^2$
	with $\alpha_1,\alpha_0 \in [0,\infty)$, $\alpha_1 >0$.
\end{theorem}
This statement is part of \cite[Theorem 6]{storath2018variational} and proved there.
We note that,
although the 
$\TGVat$ regularizer using either the Schild or the parallel transport variant of Section~\ref{sec:higherOandTGV} is lower semicontinuous (cf. \cite{holler18tgvm_mh})
 Theorem~\ref{thm:ExistenceCondR} does not apply.
The same issue occurs with pure $\TV^2$ regularization.
To overcome this, results with weaker conditions on $R$ and additional conditions on $\mathcal A$ have been established to ensure the existence of minimizers; cf. the discussion in
\cite{storath2018variational}, in particular 
 \cite[Theorem 7]{storath2018variational}. 
 The mentioned theorem applies to $\TGVat$ and pure second order TV regularizers. The conditions on $\mathcal A$ are in particular fulfilled if $\mathcal A$ is such that the data term fulfills the 
 (significantly stronger) coercivity type condition 
 \begin{align}\label{eq:coercData}
 \sum\nolimits_{i=1}^K\dist\left(\mean(\mathcal A_{i,\cdot},u^{n}),f_i\right)^q \to \infty, 
\qquad \text{as}\qquad
 \diam \left(u^{n}\right) \to \infty.
\end{align}
This coercivity type condition is for instance fulfilled if $\mathcal A$
fulfills the manifold analogue of lower boundedness, see \cite{storath2018variational} for details.
Furthermore, the conditions hold if the underlying manifold is compact. As a result we formulate the following theorem.

\begin{theorem}\label{thm:ExistenceMinimizersSecondOrderCompact}
	Assume that either $\mathcal M$ is a compact manifold, 
	or assume that $\mathcal A$ fulfills the coercivity type condition 
	\eqref{eq:coercData}. 
	Then, the inverse problem \eqref{eq:ManiInvPro} for data living in $\mathcal M^K$ with
	$\TGVat$ regularization using either the Schild or the parallel transport variant of Section~\ref{sec:higherOandTGV} has a minimizer.
	The same statement applies to (pure) second order $\TV^2$ regularization.
\end{theorem}
The part of Theorem~\ref{thm:ExistenceMinimizersSecondOrderCompact} 
concerning compact manifolds $\mathcal M$ 
is the statement of \cite[Corollary 1]{storath2018variational},
the part concerning the coercivity type condition is a special case of \cite[Theorem 8, Theorem 9]{storath2018variational}.

\subsection{Algorithmic Realization}
\label{sec:InvProbalgo}

We consider the numerical solution of \eqref{eq:ManiInvPro}. 
For differentiable data terms ($q>1$), we build on the concept of a generalized forward backward-scheme.
In the context of DTI, such a scheme has been proposed in  \cite{baust2016combined}.
We discuss an extension by a trajectory method and a Gau\ss-Seidel type update scheme
which significantly improves the performance compared to the basic scheme.

\noindent \textbf{Basic Generalized Forward Backward Scheme.} 
We denote the functional in \eqref{eq:ManiInvPro} by $\mathcal F$ and decompose it into the data term  $\mathcal D$ and the regularizer $R_\alpha$ which we further decompose into data atoms $(\mathcal D_i)_i$ and regularizer atoms $(R_\alpha)_k,$ i.e.,
\begin{align}\label{eq:DecomposeDataReg}
\mathcal F(u) = \mathcal D (u) +   R_\alpha(u) = 
\sum\nolimits_{i=1}^K  \mathcal D_i (u)  +  \sum\nolimits_{l=1}^{L}  (R_\alpha)_l(u)
\end{align}
with 
$
\mathcal D_i (u)    := \frac1q\dist(\mean(\mathcal A_{i,\cdot},u),f_i)^q,	
$
for $i=1,\ldots,K.$ Examples for decompositions 
$R_\alpha(u) = \sum\nolimits_{l=1}^{L}  (R_\alpha)_l(u)$ of $\TV$ and $\TGVat$ regularizers are given 
in Section~\ref{sec:TV} and in Section~\ref{sec:higherOandTGV}, respectively.

The basic idea of a generalized forward-backward scheme is to perform a gradient step for the explicit term, here $\mathcal D$,  as well as a proximal mapping step for each atom of the implicit term, here $(R_\alpha)_l$. 
(Concerning the computation of the corresponding proximal mappings for the $\TV$ and $\TGVat$ regularizers
of Sections~\ref{sec:TV} and \ref{sec:higherOandTGV}, we refer to these sections.) 
We now focus on the data term $\mathcal D.$
The gradient of $\mathcal D$ w.r.t. the variable $u_j,$ $j \in \{1,\ldots,N\},$ 
decomposes as
\begin{align}
\nabla_{u_j} \mathcal D (u) =    \sum\nolimits_{i=1}^K  \nabla_{u_j} \mathcal D_i (u).
\end{align} 
The gradient of $\mathcal D_i$ w.r.t.\ $u_j$ 
can then be computed rather explicitly using Jacobi fields. 
Performing this computation is a central topic of the paper \cite{storath2018variational}.
A corresponding result is \cite[Theorem 11]{storath2018variational}. 
The overall algorithm is summarized in Algorithm \ref{alg:fb_splitting}.
\begin{algorithm}[t]
  \begin{algorithmic}[1]
\State \textbf{FBS}($u^0,(\lambda_k)_k$)

\State $k=0$, 
\State \quad \textbf{repeat} until stopping criterion fulfilled 
\State \qquad  \textbf{for} $j=1,\ldots,N$
\State  \qquad \qquad $u_j^{k+0.5} = 
 \exp_{u_j^{k}} \left(-\lambda_k 
 \sum\nolimits_{i=1}^K \nabla_{u_j} \mathcal D_i \left(u^{k}\right)\right)$ 
\State \qquad  \textbf{for} $l=1,\ldots,L$
\State  \qquad  \qquad  $u^{k+0.5+ l/2{L}} =  \prox_{\lambda_k (R_\alpha)_l}( u^{k+0.5+ (l-1)/2{L}})$

\State \qquad  $k\gets k+1$
\State \Return{$u^k$}
\end{algorithmic}
\caption{FB-splitting for solving $\min_u \mathcal D(u) +  R_\alpha(u)$\label{alg:fb_splitting}}
\end{algorithm}
Note that there, for the explicit gradient descend part, we use the $k$th iterate 
$u^k=(u_1^k,\ldots,u_N^k)$ as base point for computing the gradients w.r.t. all data atoms $\mathcal D_i,$ $i=1,\ldots,K$ and all items $u_j,$ $j \in \{1,\ldots,N\}.$
This corresponds to a Jacobi type update scheme.  
During the iteration, the parameter $\lambda_k>0$ is decreased fulfilling $\sum_k \lambda_k = \infty$
and  
$\sum_k \lambda_k^2 < \infty.$ 
Recall that, for the regularizers $R_\alpha=\alpha\TV$ and 
$R_\alpha=\TGVat$ using either the Schild or the parallel transport variant of Section~\ref{sec:higherOandTGV},
the computation of line 6 in Algorithm \ref{alg:fb_splitting} can be parallelized as explained in
Section~\ref{sec:TV} and Section~\ref{sec:higherOandTGV}.

\noindent \textbf{A Generalized Forward Backward Scheme with Gau\ss-Seidel Update and a Trajectory Method.}
A well-known issue when considering gradient descent schemes is to find a suitable parameter choice for the $(\lambda_k)_k$.
Often a step size control based on line search techniques is employed. 
Above, there are two particular issues when employing an adaptive step size strategy: 
First, a single data atom $\mathcal D_{i'}$ may require a low step size whereas the other $\mathcal D_i$ would allow for much larger steps, but in the standard form one has to use the small step size for all $\mathcal D_i$.
Second, a small stepsize restriction from a single $\mathcal{D}_{i'}$  also yields a small stepsize in the the proximal mapping for the regularization terms.
Together, a small step size within an atom of the data term results in a small step size for the whole loop of the iteration Algorithm \ref{alg:fb_splitting}.

In order to overcome these step size issue, the paper \cite{storath2018variational} proposes to employ a Gauss-Seidel type update scheme together with a trajectory method. To explain the idea, we first replace the update of lines 4/5 of Algorithm \ref{alg:fb_splitting} by
\begin{equation} \label{eq:algFBwithGSupdate}  \left\{
\begin{aligned}
&\text{ for } i=1,\ldots,K\\
&\quad \text{ for }j=1,\ldots,N   \\
& \quad \qquad u_j^{k+i/2K} =  
\exp_{u_j^{k + (i-1)/2K}} \left(-\lambda_k \nabla_{u_j} \mathcal D_i (u^{k + (i-1)/2K})\right).
\end{aligned} \right. 
\end{equation}
Here, the computation of the gradients is performed in a cyclic way w.r.t.\ the $\mathcal D_i$ 
which corresponds to a Gau\ss-Seidel type update scheme.
This in particular has the following advantage:  
if we face a small step size for a particular $\mathcal D_{i'},$ 
instead of decreasing the step size for the whole loop, we may employ the following
{\em trajectory method}.
Instead of using a single geodesic line for the decay w.r.t.\ the atom $\mathcal D_{i}$ at iteration $k$,
we follow a polygonal geodesic path. That is, at iteration $k$, we do not only carry out a single but possibly multiple successive gradient descent steps w.r.t. $\mathcal D_{i}$, where the length of each step is chosen optimal for the currenct direction for $\mathcal D_i$ (by a line search strategy) and the descent steps are iterated until the sum of the step ``times'' for $\mathcal D_i$ reaches $\lambda_k$.
Details can be found in \cite{storath2018variational}.
This way, a global step size choice with all atoms (potentially negatively) influencing each other, is replaced by an autonomous step size choice for each atom.  
We denote the resulting operator by $\traj_{\lambda_k \mathcal D_{i}}$ for a data atom $\mathcal D_{i}.$
The overall algorithm is subsumed in Algorithm \ref{alg:fb_splitting_trajectory}.

\begin{algorithm}[t]
  \begin{algorithmic}[1]
\State \textbf{FBSTraj}($u^0,(\lambda_k)_k$)

\State $k=0$, 
\State \quad \textbf{repeat} until stopping criterion fulfilled

\State \qquad \textbf{for} $i=1,\ldots,K$

\State  \qquad \quad
$u^{k+i/2K} =  
  \traj_{\lambda_k \mathcal D_{i}}  \left(u^{k + (i-1)/2K}\right)$
\State \qquad  \textbf{for} $l=1,\ldots,L$
\State  \qquad  \quad 
$u^{k+0.5+ l/2{L}} =  \prox_{\lambda_k (R_\alpha)_l}( u^{k+0.5+ (l-1)/2{L}})$
\State \qquad  $k\gets k+1$
\State \Return{$x^k$}
\end{algorithmic}
\caption{FB-splitting for solving $\min_u \mathcal D(u) + R_\alpha(u)$
	using a trajectory method\label{alg:fb_splitting_trajectory}}
\end{algorithm}

We point out that also a stochastic variant of this scheme where the atoms are choosen in a random order has been proposed  in \cite{storath2018variational}. Finally, we point out that it is also possible to employ a CPPA or a PPPA as explained in Section~\ref{sec:TV}. 
This is in particular important if the data term is not differentiable, i.e., if $q=1.$ 
For details on computing the proximal mappings of the atoms $\mathcal D_{i}$ we refer to 
the paper \cite{storath2018variational}.

We illustrate the results of joint deconvolution and denoising of manifold-valued data in Figure~\ref{fig:S1_img}. The data consists of an $\mathbb S^1$-valued image convolved with a Gaussian kernel and corrupted by von Mises noise. We employ $\STGVat$ regularized deconvolution and observe good denoising and deblurring capabilities.

\begin{figure}[!tp]
	
	\def\figfolderA{experiments/inv/}
	\def\hs{\hspace{0.03\textwidth}}
	\def\vs{\vspace{0.03\textwidth}}
	\def\figurewidth{0.28\textwidth}
	\centering
	\includegraphics[width=\figurewidth,height=\figurewidth]{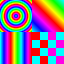} \hs
	\includegraphics[width=\figurewidth,height=\figurewidth]{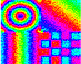}\hs 
	\includegraphics[width=\figurewidth,height=\figurewidth]{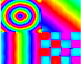} \\[2ex]

	\caption{
		Deconvoling an $\mathbb S^1$-valued image (visualized as hue values.)
		 As input data ({\em center}) we use the ground truth ({\em left}) convolved with a Gaussian kernel ($5\times 5$ kernel with $\sigma = 1$) and corrupted by von Mises noise.
		 We observe the denoising and deblurring capabilities of TGV regularized deconvolution ({\em right}.)
	}
	\label{fig:S1_img}
\end{figure}

\section{Wavelet Sparse Regularization of Manifold Valued Data}
\label{sec:WavSparse}

In contrast to TV, higher order TV type and Mumford-Shah regularizers which are all based on differences (or derivatives in the continuous setting), we here consider a variational scheme employing manifold valued interpolatory wavelets in the regularizing term. In particular, we consider a sparsity promoting $\ell^1$ type term as well as an $\ell^0$ type term.  
We obtain results on the existence of minimizers for the proposed models.
We provide algorithms for the proposed models and show the potential of the proposed algorithms.

Interpolatory wavelet transforms for linear space data have been investigated by D.~Donoho in \cite{donoho1992interpolating}.
Their analogues for manifold-valued data have been introduced by Ur Rahman, Donoho and their coworkers in \cite{rahman2005multiscale}. 
Such transforms have been analyzed and developed further in \cite{grohs2009interpolatory,grohs2010stability,weinmann2012interpolatory}. 
Typically, the wavelet-type transforms employ an (interpolatory) subdivision scheme
to predict the signal on a finer scale. 
The `difference' between the prediction and the actual data on the finer scale 
is realized by vectors living in the tangent spaces of the predicted signal points which point to the actual signal values, 
i.e., they yield actual signal values after application of a retraction such as the exponential map. 
These tangent vectors then serve as detail coefficients.
Subdivision schemes for manifold-valued data have been considered in  
\cite{wallner2005convergence, grohs2008smoothness,xie2008smoothness,weinmannConstrApprox, wallner2011convergence}.
Interpolatory wavelet transforms and subdivision are discussed in more detail in \cite{wallner2019geometric_mh}. 
All the above approaches consider explicit schemes, i.e., the measured data is processed in a forward way using the analogues of averaging rules and differences in the manifold setting. 
In contrast, we here consider an implicit approach based on a variational formulation.

\subsection{Model}
\label{sec:WavSparsemodel}

We discuss a model for wavelet sparse regularization for manifold-valued data.
For the reader's convenience, we consider the univariate situation here. For the multivariate setup and further details we refer to \cite{storath2018wavelet}.
Let $f \in \mathcal M^K$ be data in the complete, finite dimensional Riemannian manifold $\mathcal M.$ 
We consider the problem 
\begin{align}\label{eq:VarProblemIntro}
\argmin_{u \in \mathcal M^N }  \  \frac1q\dist(\mathcal A(u),f)^q + \mathcal W_\alpha^{\mu,p} (u).
\end{align}
Here, $u$ denotes the argument to optimize for; 
it may be thought of as the underlying signal 
generating the response $\mathcal A(u)\in \Mc^K$,
where $\mathcal A $ is an operator which models a system's response, for instance.
In case of pure denoising, $\mathcal A$ is the identity on $\mathcal M^N,$ $N=K.$
Further instances of $\mathcal A$ are the manifold valued analogues of convolution operators
as pointed out in Section~\ref{sec:InvProb}.  
The deviation of $\mathcal A(u)$ 
from $f$ is quantified by $\frac1q\dist(\mathcal A(u),f)^q = \frac1q\sum_{i=1}^{K}\dist(\mathcal A(u)_i,f_i)^q.$
Further,  $\alpha = (\alpha_1,\alpha_2)$ is a parameter vector regulating the trade-off between the data fidelity, 
and the regularizing term $\mathcal W_\alpha^{\mu,p}$ which is the central topic of this section
and is given by  
\begin{equation}\label{eq:DefWavRegMult}
\mathcal W_\alpha^{\mu,p}(u) = 	
\alpha_1 \cdot  \sum_{n,r} 2^{r p \left(\mu+\tfrac{1}{2}-\tfrac{1}{p} \right)}\| d_{n,r}(u) \|^p_{\hat u_{n,r}} 
+ \alpha_2 \cdot \sum_n \dist(\tilde u_{n-1,0},\tilde u_{n,0})^p.
\end{equation}
We discuss the regularizing term \eqref{eq:DefWavRegMult} in more detail in the following.
We start with the so-called detail coefficients  $d_{n,r}$ which requires some space.
The details $d_{n,r}$ at scale $r$ of the interpolatory wavelet transform for manifold valued data are given by 
\begin{equation}\label{eq:DetailInterpolMani}
d_{n,r} =  d_{n,r}(u) =  2^{-r/2}\left( \tilde u_{n,r}  \ominus  \hat u_{n,r} \right),
\qquad \hat u_{n,r} = \mathrm{S} \tilde u_{n,r-1}.  
\end{equation}
Here $\tilde u_{n,r-1} = u_{2^{R-r+1}n}$  and $\tilde u_{n,r} = u_{2^{R-r}n}$  (with $R$ the finest level)  
denote the thinned out target $u$ at scale $r-1$ and $r,$ respectively.
The coarsest level  is denoted by  $\tilde u_{n,0} = u_{2^{R}n}.$ 
The symbol $\ominus$ takes the Riemannian logarithm of the first argument w.r.t.\ the second argument as base point.
$\mathrm{S} \tilde u_{n,r-1}$ denotes the application of an interpolatory subdivision scheme $\mathrm S$ for manifold-valued data  	 
to the coarse level data $\tilde u_{\cdot,r-1}$ evaluated at the index $n$ which serves as prediction for $\tilde u_{n,r},$ i.e.,
\begin{align}
    \mathrm{S} \tilde u_{n,r-1}  =  \mean(s_{n-2 \ \cdot} \ ,\tilde u_{\cdot,r-1}).
\end{align}
Here the mask $s$ of the subdivision scheme $S$ is a real-valued sequence such that the even as well as the odd entries sum up to $1.$
The even and the odd entries yield two sets of weights; in case of an interpolatory scheme $s_0=1$ and all other even weights equal zero. The simplest example of an interpolatory scheme is the linear interpolatory scheme for which $s_{-1}=s_{1} = 1/2$ and the other odd weights equal zero. 
Thus, in the manifold setup, the prediction of the linear interpolatory consists of the geodesic midpoint between two consecutive coarse level items. 
The linear interpolatory scheme is a particular example of the interpolatory Deslaurier-Dubuc schemes
whose third order variant is given by the coefficients $s_{-3}=s_{3} = -1/16$ as well as $s_{-1}=s_{1} = 9/16$
with the remaining odd coefficients equal to zero. 
A reference on linear subdivision schemes is the book \cite{cavaretta1991stationary}; 
for manifold-valued schemes we refer to references above.

Coming back to \eqref{eq:DetailInterpolMani}, the detail $d_{n,r}$ quantifies 
the deviation between the prediction $\mathrm{S} \tilde u_{n,r-1}$ 
and the actual $r$th level data item $\tilde u_{n,r}$ 
by  
\[
  d_{n,r} = \tilde u_{n,r} \ominus \mathrm{S} \tilde u_{n,r-1} = \exp^{-1}_{\mathrm{S} \tilde u_{n,r-1}} \tilde u_{n,r}
\] 
which denotes the tangent vector sitting in $\hat u_{n,r} = \mathrm{S} \tilde u_{n,r-1}$ 
pointing to $\tilde u_{n,r}.$ 

With this information on the details $d_{n,r},$ we come back to the definition of the regularizer in \eqref{eq:DefWavRegMult}.
We observe that the symbol 
$\| \cdot \|_{\hat u_{n,r}}$ denotes the norm 
induced by the Riemannian scalar product in the point 
$\hat u_{n,r},$ which is the point where the detail $d_{n,r}(u)$
is a tangent vector at; it measures the size of the detail.
The parameter $\mu$ is a smoothness parameter and the parameter $p\geq 1$
stems from a norm type term. 
The second term addresses measures the $p$th power of the distance between neighboring items on the coarsest scale.  
 
We emphasize that the case $p=1,\mu=1$ in \eqref{eq:VarProblemIntro},
corresponds to  the manifold analogue of the LASSO \cite{tibshirani1996regression, chambolle1998nonlinear}  or $\ell^1$-sparse regularization which, in the linear case, is addressed by (iterative) soft thresholding \cite{donoho1995noising}. This case is particularly interesting since it promotes solutions $u$ which are likely to be sparse w.r.t.\ the considered wavelet expansion.

The manifold analogue of $\ell^0$-sparse regularization which actually measures sparsity is obtained by using the regularizer
\begin{equation}\label{eq:Defl0SparseMult}
\mathcal W_\alpha^0(u) = 	\alpha_1  \ \#   \{(n,r) \ : \ d_{n,r}(u) \neq 0\}
\ + \ \alpha_2 \ \#  \{n \ : \ \tilde u_{n-1,0} \neq \tilde u_{n,0}  \}.
\end{equation}
The operator $\#$ is used to count the number of elements in the corresponding set. 
Note that this way the number of non-zero detail coefficients 
of the wavelet expansion is penalized.
Similar to the linear case 
\cite{weaver1991filtering,donoho1995noising,chambolle1998nonlinear},
potential applications of the considered sparse regularization techniques are denoising and compression.

Concerning the existence of minimizers, we have the following results.
\begin{theorem}\label{thm:ExistenceStandardProbLambda2neq0}
	The variational problem \eqref{eq:VarProblemIntro} 
	of wavelet regularization using the regularizers 
	$\mathcal W_\alpha^{\mu,p}$ of \eqref{eq:DefWavRegMult} with $\alpha_2 \neq 0$ has a minimizer.
\end{theorem}
Similar to the existence results in Section~\ref{sec:InvProb} these results are based on showing lower semicontinuity and a coercivity type condition in the manifold setting. 
To ensure a coercivity type condition when $\alpha_2 = 0$ we need to impose additional conditions on $\mathcal A.$ For a precise discussion of this point we refer to \cite{storath2018wavelet}. 
As in Section~\ref{sec:InvProb} we here state a special case which is easier to access. 
\begin{theorem}\label{thm:ExistenceMinimizersWavRegCompact}
	Let $\mathcal M$ be a compact manifold, 
	or assume that $\mathcal A$ fulfills the coercivity type condition 
	\eqref{eq:coercData}. 
	The variational problem \eqref{eq:VarProblemIntro} 
	of wavelet regularization using the regularizers 
	$\mathcal W_\alpha^{\mu,p}$ of \eqref{eq:DefWavRegMult} with $\alpha_2 = 0$ has a minimizer.
\end{theorem}
\begin{theorem}\label{thm:ExistenceL0Sparse}	
	We make the same assumptions as in Theorem~\ref{thm:ExistenceMinimizersWavRegCompact}.  
	Then  wavelet sparse regularization using the $\ell^0$ type regularizing terms $\mathcal W_\alpha^0(u)$ of \eqref{eq:Defl0SparseMult} has a minimizer. 
\end{theorem}
For proofs of these theorems  
(whereby Theorem ~\ref{thm:ExistenceMinimizersWavRegCompact} is a special case of \cite[Theorem 4]{storath2018wavelet}) we refer to \cite{storath2018wavelet}.

\subsection{Algorithmic Realization}
\label{sec:WavSparsealgo}

We decompose the regularizer $W_\alpha^{\mu,p}$ into atoms $\mathcal R_k$ with a enumerating index $k$
by
\begin{equation}\label{eq:DefWavRegMultDecon}
\mathcal R_k = 	
\alpha_1 \sum_{n,r} 2^{r p \left(\mu+\tfrac{1}{2}-\tfrac{1}{p} \right)}\| d_{n,r}(u) \|^p_{\hat u_{n,r}}, \qquad \text{ or } \quad
\mathcal R_k = \alpha_2   \dist(\tilde u_{n-1,0},\tilde u_{n,0})^p,
\end{equation}
and the data term into atoms $\mathcal D_k$ according to \eqref{eq:DecomposeDataReg}.
To these atoms  we may apply the concepts of a generalized forward backward-scheme with Gauss-Seidel type update and a trajectory method (explained in Section~\ref{sec:InvProb})
as well as the concept of a CPPA or a PPPA (explained in Section~\ref{sec:TV}).
To implement these schemes expressions for 
the (sub)gradients and proximal mappings of the atoms $\mathcal R_k$ based on Jacobi fields
have been derived in \cite{storath2018wavelet}. Due to space reasons, we do not elaborate on this derivation here, but refer to the mentioned paper for details.  
Similar to \eqref{eq:DefWavRegMultDecon}, we may decompose the $\ell^0$-sparse regularizer $\mathcal W_\alpha^0$ into atoms we also denote by $\mathcal D_k,$ and apply a CPPA or PPPA. For details we refer to \cite{storath2018wavelet}. 
We illustrate $\ell^1$ wavelet regularization 
by considering a joint deblurring and denoising problem for an $\mathbb S^2$-valued time series in Figure~\ref{fig:l1}. The noisy data is convolved with the manifold-valued analogue of a discrete Gaussian kernel. As prediction operator we employ the linear interpolatory subdivision scheme which inserts the geodesic midpoint as well as the cubic Deslaurier Dubuc scheme for manifold valued data as explained above.

\begin{figure}[!tp]
	\def\figfolderB{experiments/wav/}
	\def\hs{\hfill}
	\def\vs{\vspace{0.03\textwidth}}
	\def\figurewidth{0.3\textwidth}
	\def\figurewidthB{0.30\textwidth}
	\centering
	{
		\footnotesize
		\begin{tabular}{ccc}

		\includegraphics[width=\figurewidthB, angle = 90]{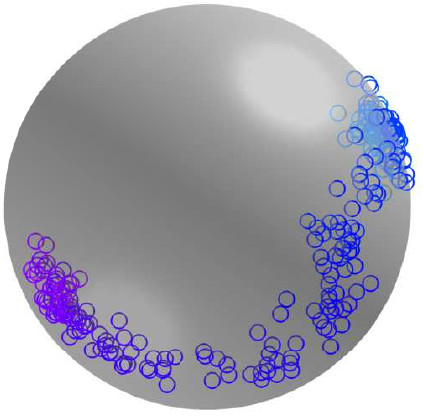}  &
		\hs\includegraphics[width=\figurewidthB, angle = 90]{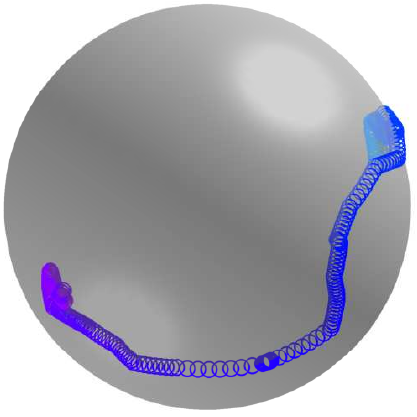}\hs&
	 	\includegraphics[width=\figurewidthB, angle = 90]{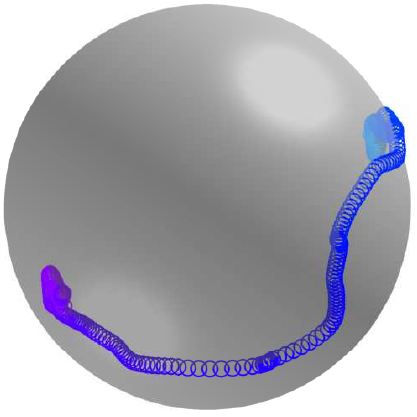}  
			\\
		\end{tabular}
	}
	\caption{
		Illustration of the proposed $\ell^1$ wavelet regularization 
		for a $\mathbb S^2$-valued time series. The given data ({\em left}) is noisy and blurred with the manifold analogue of a Gaussian kernel with $\sigma=2.$ We display the result of using 
		the first order interpolatory wavelet ({\em middle}) 
		and the third order Deslaurier-Dubuc (DD) wavelet ({\em right}).}
	\label{fig:l1}
\end{figure}

\section*{Acknowledgements}
MH acknowledges support by the Austrian Science Fund (FWF) (Grant J 4112). AW acknowledges support by the DFG Grants WE 5886/3-1 and WE 5886/4-1.

\bibliographystyle{abbrv}
\bibliography{weinmann,mh_lit_dat}

\end{document}

%% file: experiments/S2/alpha.txt
0.7

%% file: experiments/S2/kappa.txt
5.5

%% file: experiments/S2/p.txt
2

%% file: experiments/SAR/alpha_2.txt
0.32

%% file: experiments/SAR/alpha_1.txt
0.60

%% file: experiments/compare_with_TV_2D_S2_5/deltaSNR_TV.txt
5.7

%% file: experiments/compare_with_TV_2D_S2_5/deltaSNR_TV2.txt
8.4

%% file: experiments/compare_with_TV_2D_S2_5/deltaSNR_TGV.txt
8.9